\documentclass[reqno]{amsart}

\usepackage{amsmath,amssymb,amscd, accents} \usepackage{graphicx}
\DeclareGraphicsExtensions{.eps} \usepackage{mathrsfs}
\usepackage[mathcal]{eucal}

\newtheorem{Thm}{Theorem}{\bfseries}{\itshape}
\newtheorem*{Thm*}{Theorem}{\bfseries}{\itshape}
\newtheorem{Cor}{Corollary}{\bfseries}{\itshape}
\newtheorem{Prop}[Cor]{Proposition}{\bfseries}{\itshape}
\newtheorem{Lem}[Cor]{Lemma}{\bfseries}{\itshape}
\newtheorem*{Lem*}{Lemma}{\bfseries}{\itshape}
{\bfseries}{\itshape}
{\bfseries}{\itshape}
\newtheorem{Def}[Cor]{Definition}{\bfseries}{\rmfamily}
\newtheorem{Ex}[Cor]{Example}{\scshape}{\rmfamily}
\newtheorem{Rem}[Cor]{Remark}{\scshape}{\rmfamily}
{\bfseries}{\itshape}

\renewcommand\ge{\geqslant} \renewcommand\le{\leqslant}
\let\tildeaccent=\~ \let\hataccent=\^
\renewcommand\~[1]{\widetilde{#1}}

\def\<{\left<} \def\>{\right>} \def\({\left(} \def\){\right)}

\def\abs#1{\left\vert #1 \right\vert} \def\norm#1{\left\Vert #1
  \right\Vert} 

\let\parasymbol=\S \def\secref#1{\parasymbol\ref{#1}}
 \def\pd#1#2{\tfrac{\partial#1}{\partial#2}}

\let\polishL=l \def\Zoladek.{\.Zol\c adek}

 \def\const{\operatorname{const}}
 \def\Mat{\operatorname{Mat}}
 
 \def\Im{\operatorname{Im}}
 \def\dist{\operatorname{dist}}
 
\def\GL{\operatorname{GL}} \def\SL{\operatorname{SL}}
\def\etc.{\emph{etc}.}
 \def\Sing{\operatorname{Sing}}

\def\:{\colon} \def\R{{\mathbb R}} \def\C{{\mathbb C}} \def\Z{{\mathbb
    Z}} \def\N{{\mathbb N}} \def\Q{{\mathbb Q}} \def\P{{\mathbb P}}
\def\H{{\mathbb H}}

\def\A{{\mathbb A}}
\def\K{{\mathbb K}}

\let\PolishL=\L 
\def\L{{\mathbb L}}

 \def\e{\varepsilon} \def\S{\varSigma}
   
 \def\diag{\operatorname{diag}}
\def\poly{\operatorname{poly}}

 \def\Lojas.{\PolishL ojasiewicz}
 
\def\cP{{\mathcal P}} \def\cR{{\mathcal R}}
\def\scS{{\mathscr S}} \def\scC{{\mathscr C}}
\def\scP{{\mathscr P}} \def\scD{{\mathscr D}}
\def\scM{{\mathscr M}} \def\scR{{\mathscr R}}
 \def\cL{{\mathcal L}} \def\cR{{\mathcal R}}
  
 \def\cD{{\mathcal D}}

 \def\mult{\operatorname{mult}}

\def\rest#1{{\vert_{#1}}}

\def\supp{\operatorname{supp}}
\def\Gal{\operatorname{Gal}}

\def\w{\omega}

\def\Qa{\Q^{\text{alg}}}
\def\O{\mathcal{O}}
\def\Ob{\overline{\O}}

\def\LT{\operatorname{LT}}
\def\height{\operatorname{ht}}
\def\fp{{\mathfrak p}}
\def\fq{{\mathfrak q}}

\begin{document}

\title{Zero counting and invariant sets of differential equations}

\author{Gal Binyamini}\address{University of Toronto, Toronto, 
Canada}\email{galbin@gmail.com}

\begin{abstract}
  Consider a polynomial vector field $\xi$ in $\C^n$ with algebraic
  coefficients, and $K$ a compact piece of a trajectory. Let $N(K,d)$
  denote the maximal number of isolated intersections between $K$ and
  an algebraic hypersurface of degree $d$. We introduce a condition on
  $\xi$ called \emph{constructible orbits} and show that under this
  condition $N(K,d)$ grows polynomially with $d$. 

  We establish the constructible orbits condition for linear
  differential equations over $\C(t)$, for planar polynomial
  differential equations and for some differential equations related
  to the automorphic $j$-function.

  As an application of the main result we prove a polylogarithmic
  upper bound for the number of rational points of a given height in
  planar projections of $K$ following works of Bombieri-Pila and
  Masser.
\end{abstract}
\maketitle
\date{\today}

\section{Introduction}
\label{sec:intro}

Let $\xi$ be a polynomial vector field in $\C^n$ defined over an
algebraically closed field $\K\subset\C$,
\begin{equation} \label{eq:main-vf}
  \xi = \sum_{i=1}^n \xi_i(x) \pd{}{x_i}, \qquad \xi_i\in\K[x_1,\ldots,x_n].
\end{equation}
We denote by $\Sing\xi$ the singular locus of $\xi$, i.e. the set of
common zeros of $\xi_i$. For $p\in\C^n$, we define the \emph{orbit} of
$p$, denoted $\O_p$, to be the leaf of the (singular) complex
foliation determined by $\xi$. Equivalently $\O_p$ is the minimal
$\xi$-invariant set containing $p$. We remark that if $p\in\Sing\xi$
then $\O_p=\{p\}$. A subset of $\C^n$ is said to be $\xi$-invariant if
it is a union of orbits. We define the \emph{orbit closure} at $p$,
denoted $\Ob_p$, to be the Zariski closure of $\O_p$. Equivalently,
$\Ob_p$ is the minimal Zariski closed set containing $\O_p$.

\begin{Def}\label{def:const-orbits}
  Let $V\subset\C^n$ be a $\xi$-invariant variety defined over $\K$.
  We say that $\xi$ admits (or has) \emph{constructible orbits} in $V$ if the
  relation
  \begin{equation}\label{eq:orbit-rel}
    E\subset V\times V, \qquad E:=\{(p,q)\in V\times V : q\in\Ob_p\}
  \end{equation}
  is $\K$-constructible\footnote{See
    Remark~\ref{rem:const-orbits-field} concerning constructibility
    over $\K$ vs. $\C$.} (below \emph{constructible set} always means
  over $\K$ unless otherwise stated).
\end{Def}

The notion of differential equations with constructible orbits is
motivated by Nesterenko's work on E-functions and in particular by the
paper \cite{nesterenko:galois}. For more background
see~\secref{sec:linear}. In~\secref{sec:examples} we establish the
constructible orbits condition for systems of linear differential
equations over $\C(t)$; for planar polynomial differential equations;
and for some differential equations related to the automorphic
$j$-function.

We now turn to the description of our main result. Let $D_r$ (resp.
$\bar D_r$) denote the open (resp. closed) disc of radius $r$ centered
around the origin in $\C$. If $r$ is omitted $r=1$ is assumed. When
we speak of holomorphic functions on non-open set we always mean that
the function is holomorphic in a neighborhood of the set.
\begin{Def}
  A holomorphic map $\phi:\bar D\to V$ is said to be a parametrized
  (singular) trajectory of $\xi$ if $\phi(\bar D)\not\subset\Sing\xi$
  and if for every $z\in\bar D$ the complex vectors $\phi'(z)$ and
  $\xi(\phi(z))$ are complex proportional when they are both non-zero.
\end{Def}
Note that we do not require $z$ to act as the natural time parameter
with respect to the $\xi$-flow in the definition of parameterized
trajectories. In particular we allow $\phi(\bar D)$ to pass through
singular points of $\xi$ as long as it remains holomorphic. We define
the \emph{transcendence degree} $\kappa(\phi)$ to be the dimension of
the Zariski closure of $\phi(\bar D)$. The following is our main
result.

\begin{Thm} \label{thm:main} Let $V,\xi$ be defined over $\K=\Qa$ and
  $V$ invariant under $\xi$. Suppose that $\xi$ admits constructible
  orbits in $V$. Let $\phi:\bar D\to V$ be a parametrized trajectory
  of $\xi$. Then there exists a constant $C_\phi$ with the following
  property: For every $P\in\C[x_1,\ldots,x_n]$ such that
  $P\circ\phi\not\equiv0$,
  \begin{equation}\label{eq:main-estimate}
    \#\{z\in\bar D:P(\phi(z))=0\} \le C_\phi\cdot d^{2\kappa(m+1)}\log d
  \end{equation}
  where $d=\deg P$, $\kappa=\kappa(\phi)$ and $m=\dim V$.
\end{Thm}

For example, whenever a system of a type considered
in~\secref{sec:examples} is defined over $\K=\Qa$,
Theorem~\ref{thm:main} applies to produce a polynomial estimate (in
$d$) for the number of intersections between a parametrized trajectory
of the system and an algebraic hypersurface of degree $d$.

For linear differential equations over $\C(t)$, Novikov-Yakovenko
\cite{ny:poly-env} give a single exponential (in $d$) upper bound for
left hand side of~\eqref{eq:main-estimate}. Since linear differential
equations always admit constructible orbits, Theorem~\ref{thm:main}
improves this to a polynomial asymptotic whenever the equation is
defined over $\Qa$. Similarly, \cite{ny:chains} gives an iterated
exponential bound for trajectories of arbitrary polynomial vector
fields, and we see that this can be improved to a polynomial estimate
when the vector field has constructible orbits and is defined over
$\Qa$. Polynomial zero estimates have significant applications that do
not follow from the previously known estimates. We demonstrate one
such application in the following subsection~\secref{sec:density}.

\subsection{Application: density of rational points on parametrized trajectories}
\label{sec:density}

For a reduced quotient $\tfrac a b$ we introduce the \emph{height}
$H(\tfrac a b)=\max(\abs a,\abs b)$. To avoid technicalities we set
$H(0)=1$. For a vector $v\in\Q^k$ we let $H(v)$ denote the maximal
height of its components. For a set $X\subset\C^k$ we denote
\begin{equation}
  X(\Q,H) := \{ v\in X\cap\Q^k : H(v)\le H\}, \qquad N(X,H):=\#X(\Q,H).
\end{equation}
The problem of estimating the number of integers or rational points of
bounded height for various types of sets $X$ has been considered by
numerous authors, starting with the work of Jarn\'ik \cite{jarnik}.
From the seminal works of Bombieri-Pila \cite{pb} and Pila-Wilkie
\cite{pw} it is known that an asymptotic estimate $N(X,H)=O(H^\e)$
holds for any $\e>0$ if $X$ is a transcendental curve definable in any
o-minimal structure (and an appropriate generalization holds for
higher dimensional sets as well). In this and all other asymptotics
discussed in this section, the asymptotic is taken with respect to $H$
for a fixed $X$.

In some cases one may hope to improve the estimate $O(H^\e)$ to a
polylogarithmic estimate $O(\log^\gamma H)$ for some constant $\gamma$
(depending on $X$). Such results have been obtained by Masser
\cite{masser:zeta} for $X$ given by a (compact piece of) the graph of
the Riemann zeta function, and subsequently by Besson \cite{besson},
Boxall-Jones \cite{bj:alg} and Jones-Thomas \cite{jt:weierstrass}
for graphs of other special functions; and by Pila \cite{pila:pfaff}
for arbitrary Pfaffian curves. In this subsection we show that a
similar polylogarithmic estimate holds for parametrized trajectories
of differential equations with constructible orbits.

We recall the following proposition from \cite{masser:zeta}. For a
finite set $S\subset\C^2$ we denote by $\w(S)$ the degree of the
minimal algebraic curve containing $S$. Note that the original
proposition allows for more refined control over various parameters,
and holds for more general number fields in place of $\Q$. We present
a simplified version sufficient for our purposes.

\begin{Prop}[\protect{\cite[Proposition~2]{masser:zeta}}]\label{prop:masser}
  Let $r>0$ and let $f_1,f_2:\bar D_{2r}\to\C$ be two holomorphic
  functions and denote $\Phi:=(f_1,f_2)$. Suppose $Z\subset\bar D_r$
  is a finite set of complex numbers such that $\Phi(z)\in\Q^2$ for
  $z\in Z$ and denote $H:=\max_{z\in Z} H(\Phi(z))$. Then
  \begin{equation}
    \w(\Phi(Z)) = O(\log H).
  \end{equation}
\end{Prop}

The following is a direct corollary.

\begin{Cor} \label{cor:density}
  Let $V,\xi$ be defined over $\K=\Qa$ and $V$ invariant under $\xi$.
  Suppose that $\xi$ admits constructible orbits in $V$. Let
  $\phi:\bar D\to V$ be a parametrized trajectory of $\xi$. Let
  $P_1,P_2\in\C[x_1,\ldots,x_n]$ and set
  \begin{equation}
    \Phi:\bar D\to\C^2, \qquad \Phi:=(P_1\circ\phi,P_2\circ\phi).
  \end{equation}
  If $\Im\Phi$ is not contained in an algebraic curve then
  \begin{equation}
    N(\Im\Phi,H) = O(\log^{2\kappa(m+1)} H \cdot \log\log H)
  \end{equation}
  where $\kappa=\kappa(\phi)$ and $m=\dim V$.
\end{Cor}
\begin{proof}
  Let $\phi$ be holomorphic in a disc $\bar D_r$ for some $r>1$.
  Covering $\bar D$ by $N$ discs of radius $(r-1)/2$ and applying
  Proposition~\ref{prop:masser} we see that
  \begin{equation}
    \{ v\in(\Im\Phi)\cap\Q^2 : H(v)\le H\} \subset \cup_{j=1,\ldots,N} C_j\cap \Im\Phi.
  \end{equation}
  where each $C_j\subset\C^2$ is an algebraic curve $C_j=\{Q_j=0\}$
  and $\deg Q_j=O(\log H)$. By Theorem~\ref{thm:main}, since we assume
  that $\Im\Phi$ is not contained in $C_j$, we have
  \begin{equation}
    \# [C_j\cap\Im\Phi] \le \#\{ z\in\bar D: [Q_j(P_1,P_2)\circ\phi](z)=0\} = O(d^{2\kappa(m+1)}\log d)
  \end{equation}
  where $d=\deg Q_j(P_1,P_2)=O(\log H)$. The statement of the
  corollary follows immediately.
\end{proof}

\subsection{Outline of the proof of Theorem~\ref{thm:main}}

Let $\xi$ be defined over $\Qa$. For simplicity we consider the case
$V=\C^n$. Fix $p\in\C^n$ a non-singular point of $\xi$ and consider a
parametrized trajectory $\phi:\bar D\to\C^n$ with $\phi(0)=p$. To
simplify the exposition suppose further that the Zariski closure of
$\Im\phi$ is $\C^n$. Our goal is to estimate, for a given polynomial
$P$, the number of zeros of $f:=P\circ\phi$ in $\bar D$. We may assume
without loss of generality that $P$ has unit norm (for instance
$L_2$-norm) in the space of polynomials of degree $d$.

Our basic zero-counting argument follows a familiar approach using the
Jensen inequality to compare the growth of $f$ to the number of its
zeros. The precise statement used is given in
Proposition~\ref{prop:jensen}. Roughly, the number of zeros is
estimated (up to a multiplicative constant) by $\log M-\log m$ where
$m$ denotes the maximum of $\abs f$ on $\bar D$, and $M$ denotes the
maximum of $\abs f$ on some slightly larger disc.

Since $\phi$ is bounded on $\bar D$ as well as a slightly larger disc,
it is easy to see that $\log M$ grows at most polynomially in $d$ (in
fact $\log M=O(d\log d)$). The main difficulty is thus in producing a
lower estimate for $\log m$ which is polynomial in $d$. By a Cauchy
estimate argument the maximum of $\abs f$ over a disc can be estimated
from below in terms of its derivatives at the origin. Since these
derivatives can be computed in terms of $\xi$, we reduce the problem
to the following statement (see Lemma~\ref{lem:traj-max-estimate}):
there exists $k=O(d^n)$ such that $\log \abs{\xi^k P(p)}$ admits a
lower estimate polynomial in $d$.

Let $\mu:=\mu(d)$ be some polynomial function of $d$ that will be determined later. We consider the linear map
$T$ from the space of polynomials of degree $d$ to $\C^\mu$, taking a
polynomial $P$ to the vector of its first $\mu$ derivatives at $p$. The
existence of a non-trivial kernel for $T$ is equivalent to the
vanishing of certain minors, which we call \emph{$\mu$-elimination
  minors}. Each minor $M$ is a polynomial function in $\C^n$ with
explicitly bounded degrees and coefficients. We show
(see Lemma~\ref{lem:minor-v-abs}) that if $\e:=\abs{M(p)}$ is non-zero for
one of these minors then there exists $k<\mu$ with
$\log \abs{\xi^k P(p)}\ge\log\e-\poly(d)$. It thus remains to prove
that for some suitable $M$ we have $\log\e\ge-\poly(d)$.

Using a diophantine \Lojas. inequality due to Brownawell
(see~\secref{sec:lojas}) we show that for some minor $M$, the value
$\log\e$ is, up to polynomial factors in $d$, bounded from below by
the logarithm of the distance from $p$ to the zero locus $V^d_\mu$ of
the set of minors (i.e. if all minors are small then there must in
fact be a nearby common zero). It remains to give a polynomial lower
bound for this logarithmic distance.

Results known as \emph{multiplicity estimates}
(see~\secref{sec:mult-estimates}) imply that if $\mu=Cd^n$ for a
sufficiently large constant $C\in\N$ then any polynomial having the first
$\mu$ derivatives vanishing at a point $q\in\C^n$ must in fact be
identically vanishing on the trajectory through $q$. It follows that
for this choice of $\mu$ the set $V^d_\mu$, i.e. the set of points for
which $T$ admits a non-trivial kernel, is contained in the union of
all the trajectories that satisfy some non-trivial polynomial
relation. Call this set $B$. In particular, $p\not\in B$ by
assumption. In general the set $B$ may be extremely complicated.
However, we claim that for systems with constructible orbits $B$ is a
Zariski closed set. In particular it follows that since $p\not\in B$
the distance between $p$ and $B$, and hence the distance between $p$
and $V^d_\mu$, is lower bounded by a constant independent of $d$. This
concludes the proof.

To see that $B$ is closed, we note that since the orbit-closure
relation $q\in\Ob_p$ is constructible its $p$-fibers have uniformly
bounded degrees. It follows that there is some constant $N$ such that
any trajectory satisfying some non-trivial polynomial relation must in
fact satisfy a polynomial relation of degree at most $N$. The set of
points where the trajectory satisfies such a relation for a fixed
degree $N$ is Zariski closed, and is in fact given by $V^d_\mu$ for
$d=N$ and $\mu$ suitably chosen as above.

To conclude we make some remarks concerning the general case.
Replacing $\C^n$ by a $\Qa$-variety $V$ does not introduce essential
difficulties: one simply replaces the general polynomial ring by the
coordinate ring of $V$. However, if the trajectory through $p$
satisfies an algebraic relation then the construction outlined above
cannot be carried out verbatim (as we would have $p\in V^d_\mu$ for
any sufficiently large $d$). One may be tempted to replace $V$ in this
case by the Zariski closure of the trajectory, but since this variety
is not necessarily defined over $\Qa$ this would preclude our use of a
diophantine \Lojas. inequality.

Instead, we rely on the theory of Gr\"obner basis. Namely we show that
one can restrict to a variety $V'$ \emph{defined over $\Qa$} such that
the ideals of definition of the orbits $\Ob_q$ all share the same
Gr\"obner diagram, for $q$ in some open neighborhood $U$ of $p$. This
essentially uniformizes the orbits in a neighborhood of $p$. One can
then construct a set of $\mu$-eliminating minors analogous to those
considered above, whose set of common zeros does not intersect $U$.
The rest of the proof then proceeds essentially unchanged.

\section{Preliminaries and constructible orbits}
\label{sec:prelim}

In this section we recall some preliminary results and prove some
technical statements needed for the main argument. We fix $\xi,V$ as
in~\secref{sec:intro}. When we use the asymptotic notation $O(\cdot)$
the constants may depend on $\xi,V$.

We introduce some notations. We denote by $\cP_d$ (resp.
$\cP_{\le d}$) the set (resp. vector space) of polynomials of (total)
degree $d$ (resp. at most $d$) in $\C[x_1,\ldots,x_n]$. If
$W\subset\C^n$ is an algebraic variety we denote its ideal of
definition by $I_W\subset\C[x_1,\ldots,x_n]$. If $S$ is a collection
of polynomials we denote its common zero locus by $Z(S)$. We summarize
some basic properties of $\xi$-invariance below.

\begin{itemize}
\item An algebraic variety $W$ is $\xi$-invariant if and only if its
  ideal $I_W$ is stable under $\xi$, where we view $\xi$ as a
  derivative of $\C[x_1,\ldots,x_n]$. Indeed, suppose $W$ is
  $\xi$-invariant and let $P\in I_W$. Then $P$ vanishes identically on
  $W$, and hence so does its derivative along $\xi$, i.e.
  $\xi P\in I_W$. Conversely, suppose that $I_W$ is stable under
  $\xi$, fix $p\in W$ and let $P\in I_W$. Then $\xi^kP\in I_W$ and in
  particular $\xi^k P(p)=0$ for every $k\in\N$. By analytic
  continuation we deduce that $P$ vanishes identically on $\O_p$,
  i.e. $\O_p\subset W$ and hence $W$ is $\xi$-invariant.
\item The Zariski closure $\bar X$ of a $\xi$-invariant set $X$ is
  $\xi$-invariant. Indeed, if $P\in I_{\bar X}=I_X$ then $P$ vanishes
  identically on $X$, and since $X$ is invariant under $\xi$ it
  follows that $\xi P$ also vanishes identically on $X$, i.e.
  $\xi P\in I_X=I_{\bar X}$.
\item If $V$ is a $\xi$-invariant variety and $V_1,\ldots,V_k$ are its
  irreducible components then each $V_j$ is $\xi$-invariant. For
  instance let $V_1'=V_1\setminus(V_2\cup\cdots\cup V_k)$, which is a
  Zariski-dense subset of $V_1$. Then $V_1'$ is open in $V$ and it
  follows that for every $p\in V_1'$, the germ of its $\xi$-orbit is
  contained in $V_1'\subset V$. Thus if $P\in I_{V_1}$ then $\xi P$
  vanishes on $V_1'$ and hence also on its Zariski closure $V_1$, i.e.
  $\xi P\in I_{V_1}$. In particular it follows that the orbit closures
  $\Ob_p$ are irreducible.
\end{itemize}

\subsection{Multiplicity estimates}
\label{sec:mult-estimates}

We recall some known estimates on the multiplicity of a polynomial
restricted to the trajectory of a polynomial vector field. The
following estimate from \cite{me:mult-morse} (improving upon similar
results of \cite{gabrielov:mult}) is sharp with respect to $d$ and has
explicit constants.

\begin{Thm}[\protect{\cite[Corollary~1]{me:mult-morse}}]\label{thm:mult-morse}
  Let $p\in \C^n\setminus\Sing\xi$ and denote by $\gamma_p$ the (germ
  of the) trajectory of $\xi$ through $p$. Let $P\in\cP_{\le d}$. If
  $P\rest{\gamma_p}\not\equiv0$ then
  \begin{equation}
    \mult_p P\rest{\gamma_p} \le 2^{n+1} (d+(n-1)\delta)^n
  \end{equation}
  where $\delta=\deg\xi$.
\end{Thm}

However, if one allows the constants to depend on the trajectory
$\gamma_p$ then essentially better estimates (with respect to $d$) can
be obtained in the case that $\gamma_p$ satisfies some non-trivial
polynomial relations. The following result appeared in
\cite{nesterenko:mult-nonlinear} (and later
\cite{me:mult-sing,me:mult-morse}).

\begin{Thm}[\protect{\cite[Theorem~1]{nesterenko:mult-nonlinear}}] \label{thm:mult-refined}
  Let $p\in \C^n\setminus\Sing\xi$ and denote by $\gamma_p$ the (germ
  of the) trajectory of $\xi$ through $p$. Let $\kappa$ denote the
  dimension of the Zariski closure $\overline{\gamma_p}$ of
  $\gamma_p$. There exists a constant $C_{\gamma_p}$ such that for
  every $P\in\cP_{\le d}$, if $P\rest{\gamma_p}\not\equiv0$ then
  \begin{equation}
    \mult_p P\rest{\gamma_p} \le C_{\gamma_p} d^\kappa.
  \end{equation}
  Moreover, $C_{\gamma_p}$ can be chosen to depend only on
  $\deg\overline{\gamma_p}$.
\end{Thm}

\subsection{Constructible orbits}

The following lemma is an easy consequence of
Theorem~\ref{thm:mult-morse}.

\begin{Lem}\label{lem:orb-ideal-def}
  Let $p\in V$ and $P\in\cP_{\le d}$. Then $P\in I_{\Ob_p}$ if and only if
  $P,\xi P,\ldots,\xi^\nu P$ vanish at $p$ where $\nu=C d^n$ and
  $C$ is a constant depending only on $n$.
\end{Lem}
\begin{proof}
  If $p\in\Sing\xi$ then $\Ob_p=\{p\}$ and the claim is obvious.
  Otherwise, in the notation of Theorem~\ref{thm:mult-morse} we have
  $P\in I_{\Ob_p}$ if and only if $P\rest{\gamma_p}\equiv0$ and the result
  follows from the theorem.
\end{proof}

If $I\subset\C[x_1,\ldots,x_n]$ is an ideal, we say that $I$ is
generated in degree $N$ if $I$ is generated as an ideal by
$I\cap\cP_{\le N}$. If $V\subset\C^n$ is an irreducible algebraic
variety $\deg V$ will denote its \emph{degree}, i.e. the number of
intersections with a generic affine plane of complementary dimension.
If $V$ is reducible we define $\deg V$ as the sum of the degrees of
its irreducible components.

\begin{Lem}\label{lem:deg-ideal-vs-variety}
  The ideal of an algebraic variety $V\subset\C^n$ is generated in
  some degree $N=N(\deg V)$ depending only on $\deg V$. Conversely, if
  $I$ is generated in degree $N$ then $\deg Z(I)$ is bounded by some
  $D=D(N)$ depending only on $N$.
\end{Lem}
\begin{proof}
  The second statement is a consequence of the Bezout theorem: since
  $Z(I)$ is cut-out by equations of degree at most $N$, its degree can
  be estimated in terms of $N$. For the first statement, we first note
  that $V$ is set-theoretically defined by a collection of equations
  of degree at most $\deg V$. If $V$ is irreducible (or
  pure-dimensional) then such a collection is given, for instance, by
  the canonical equations of Chow and van der Waerden
  \cite[Corollary~3.2.6]{gkz}. If $V=V_1\cup\cdots\cup V_k$ is a
  decomposition into irreducible components and $I_1,\ldots,I_k$
  denote the corresponding ideals constructed above, then the ideal
  $I_1\cdots I_k$ defines $V$ set-theoretically and is generated in
  degree $\sum_j\deg V_j=\deg V$.

  Let $I$ denote the ideal generated by equations as above. Then by
  the Nullstellensatz we have $I_V=\sqrt I$. The radical $\sqrt I$ can
  be explicitly computed from the equations generating $I$ (see
  e.g.~\cite{laplange:radical}), and in particular the degrees of the
  generators of $\sqrt I$ can be estimated in terms of $\deg V$ as
  claimed (\cite{laplange:radical} also gives an explicit estimate,
  which we do not require but may be of use for deriving an effective
  version of some of the results in this paper).
\end{proof}

We have the following alternative characterization of the notion of
constructible orbits.

\begin{Prop}\label{prop:orbit-ideal-deg}
  The following are equivalent:
  \begin{enumerate}
  \item $\xi$ has constructible orbits in $V$.
  \item The ideals $I_{\Ob_p}$ for $p\in V$ are generated in some
    degree $N$ independent of $p$.
  \item The degrees $\deg\Ob_p$ for $p\in V$ are uniformly bounded.
  \end{enumerate}
\end{Prop}
\begin{proof}
  The equivalence of conditions 2 and 3 follows from
  Lemma~\ref{lem:deg-ideal-vs-variety}.

  Next we show that 1 implies 3. Suppose that $\xi$ has
  constructible orbits in $V$. Then $\{\Ob_p:p\in V\}$ are the fibers
  of the constructible set $E$ as in~\eqref{eq:orbit-rel}. This
  already implies that the degrees of $\Ob_p$ are uniformly bounded.
  Indeed, by standard finiteness properties of the constructible
  class, there is a uniform upper bound $\nu_k$ for the number of
  isolated intersections between any affine-linear plane $L$ of
  codimension $k$ and any of the fibers $\Ob_p$ of $E$. For any fixed
  fiber $\Ob_p$ one can choose the affine planes $L_k$ such that the
  number of isolated intersections between $\Ob_p$ and $L_k$ is the
  degree of the $k$-dimensional part\footnote{In fact $\Ob_p$ is
    irreducible and in particular pure-dimensional, but we prove the
    more general statement for completeness.} of $\Ob_p$, and thus
  $\nu_1+\cdots+\nu_n$ is a uniform upper bound for the degrees of
  $\Ob_p$. We remark that the same proof holds if we assume
  that~\eqref{eq:orbit-rel} is $\C$-constructible instead of
  $\K$-constructible.

  Finally we show that 2 implies 1. Suppose the ideal $I_{\Ob_p}$
  is generated in degree $N$ for every $p\in V$. Then we have
  \begin{equation}\label{eq:q-Obp-cond}
    q\in\Ob_p \iff \forall_{P\in\cP_{\le N}} [ P\in I_{\Ob_p} \implies P(q)=0 ].
  \end{equation}
  We note that the right hand side of~\eqref{eq:q-Obp-cond} (with $N$
  fixed as above) is a first order $\K$-formula in $p,q$ and hence
  defines a constructible set, assuming we show that the condition
  $P\in I_{\Ob_p}$ (for $P$ of degree at most $N$) can be expressed by
  a $\K$-formula. For this, observe that by by
  Lemma~\ref{lem:orb-ideal-def}
  \begin{equation}
    P\in I_{\Ob_p} \iff P(p)=0,[\xi P](p)=0,\ldots,[\xi^\nu P](p)=0.
  \end{equation}
  where $\nu$ is a constant depending only on $N$ and $n$, which is
  indeed a (quantifier-free) $\K$-formula in $p$ and the coefficients
  of $P$.
\end{proof}

\begin{Rem}\label{rem:const-orbits-field}
  The proof of Proposition~\ref{prop:orbit-ideal-deg} shows that if
  the orbit closure relation is constructible over $\C$ then it is
  constructible over $\K$ (since the implication 1$\implies$3 works
  over $\C$ as well, and condition 3 does not depend on the field of
  definition). Therefore in Definition~\ref{def:const-orbits} it does
  not matter if we require constructibility over $\K$ or $\C$.
\end{Rem}

\subsection{Gr\"obner bases}

Let $\prec$ denote the degree-lexicographic ordering on the monomials
in $\C[x_1,\ldots,x_n]$. For $P\in\C[x_1,\ldots,x_n]$ we let
$\LT(P)\in\N^n$ denote the index of its leading (i.e. highest with
respect to $\prec$) monomial and $\supp P\subset\N^n$ denote the set
of indices of all monomials with non-zero coefficients in $P$. If $S$
is a set of polynomials we denote $\LT(S):=\{\LT(s):s\in S\}$. We
recall the following multi-variate division with remainder theorem.

\begin{Prop}[\protect{\cite[Theorem~2.3]{clo:ideals}}]\label{prop:grobner-division}
  Let $I\subset\C[x_1,\ldots,x_n]$ be an ideal and $P\in P_{\le d}$. Then $P$
  can be represented in the form $P=QA+R$ where
  \begin{equation}
    Q\in I,\quad A\in\C[x_1,\ldots,x_n],\quad \supp R\subset \LT(\cP_{\le d})\setminus\LT(I).
  \end{equation}
\end{Prop}

If $I\subset\C[x_1,\ldots,x_n]$ is an ideal and $S\subset I$, we say
that $S$ is a Gr\"obner basis if $\LT(I)$ is generated by $\LT(S)$ (as
an ideal in the semigroup $\N^n$).

\begin{Prop} \label{prop:diagrams}
  Suppose that $\xi$ has constructible orbits in $V$. Then the set of
  diagrams $\scD:=\{\LT(I_{\Ob_p}):p\in V\}$ is finite, and for each
  $\cD\in\scD$ the set $\scS_\cD\subset V$ given by
  \begin{equation}
    \scS_\cD := \{ p\in V : \LT(I_{\Ob_p}) = \cD \}
  \end{equation}
  is constructible.

  If $\xi,V$ are defined over $\Q$ then each $\scS_\cD$ is invariant
  under the action of $\Gal(\K/\Q)$.
\end{Prop}
\begin{proof}
  By Proposition~\ref{prop:orbit-ideal-deg} the ideals $I_{\Ob_p}$ are
  generated in some degree $N$ independent of $p$. Then there exists
  exists some $N'$, depending only on $N$, such that for every $p$ the
  ideal $I_{\Ob_p}$ has a Gr\"obner basis $S_p\subset\cP_{\le N'}$ (in fact
  $N'$ can be explicitly estimated from $N$, for instance using
  Buchberger's algorithm, though we shall not use this fact). Since
  $\LT(I_{\Ob_p})$ is generated by $LT(S_p)$ and the latter varies
  over finitely many values, we see that $\scD$ is finite. Moreover,
  given $\cD\in\scD$ we have for every $p\in V$
  \begin{equation}\label{eq:LT-Obp-cond}
    \LT(I_{\Ob_p})=\cD \iff \LT(I_{\Ob_p}\cap\cP_{\le N'})=\cD\cap\LT(\cP_{\le N'}).
  \end{equation}
  The right hand side of~\eqref{eq:LT-Obp-cond} can be expressed
  as follows
  \begin{multline}
    \big[\forall(\alpha\in\cD,|\alpha|\le N')\exists(P=x^\alpha+\sum_{\beta\prec\alpha}c_\beta x^\beta):P\in I_{\Ob_p}\big] \land \\
    \big[\forall(\alpha\not\in\cD,|\alpha|\le N')\not\exists(P=x^\alpha+\sum_{\beta\prec\alpha}c_\beta x^\beta):P\in I_{\Ob_p}\big]
  \end{multline}
  Since it was shown in the proof of
  Proposition~\ref{prop:orbit-ideal-deg} that the condition
  $P\in I_{\Ob_p}$ is expressible by a $\K$-formula, we see that indeed
  the set $\scS_\cD$ of $p$ satisfying~\eqref{eq:LT-Obp-cond} for a
  fixed $\cD$ is indeed a constructible set.

  Suppose now that $\xi,V$ are defined over $\Q$ and let
  $\sigma\in\Gal(K/\Q)$. A polynomial $P\in\K[x_1,\ldots,x_n]$
  satisfies $P\in I_{\Ob_p}$ if and only if $P,\xi P,\ldots$ all
  vanish at $p$. Applying $\sigma$ (which commutes with $\xi$ by our
  assumption) we see that this occurs if and only if
  $\sigma P,\xi(\sigma P),\ldots$ all vanish at $\sigma(p)$, i.e. if
  and only if $\sigma P\in I_{\Ob_{\sigma(p)}}$. Thus
  $I_{\Ob_{\sigma(p)}}=\sigma I_{\Ob_p}$. Since the Gr\"obner diagram is
  invariant under automorphisms of the field (depending only on the
  leading term with a non-zero coefficient) we see that
  $\LT(I_{\Ob_p})=\LT(\sigma I_{\Ob_p})$ so that indeed each strata
  $\scS_\cD$ is invariant under $\sigma$.
\end{proof}

\begin{Prop}\label{prop:V-grobner-reduct}
  Suppose that $\xi$ has constructible orbits in $V$ and let $p\in V$.
  There exists a $\xi$-invariant $\K$-variety $V'\subset V$ and a
  Zariski open dense subset $U\subset V'$ such that $\O_p\subset U$
  and $\LT(I_{\Ob_q})$ is constant for $q\in U$.

  If $\xi,V$ are defined over $\Q$ then $V'$ can be chosen to be
  defined over $\Q$.
\end{Prop}
\begin{proof}
  We may without loss of generality replace $V$ by one of its
  irreducible components. Recall that the subsets $\scS_\cD\subset V$
  of Proposition~\ref{prop:diagrams} form a finite partition of $V$
  into constructible sets. Moreover, each of the sets $\scS_\cD$ is
  $\xi$-invariant. Indeed, suppose $q\in\scS_\cD$ and $q'\in\O_q$.
  Then $\O_{q'}=\O_q$, and therefore
  $\LT(I_{\Ob_{q'}})=\LT(I_{\Ob_q})=\cD$, so $q'\in\scS_\cD$. It
  follows that the Zariski closures $\overline{\scS_\cD}$ are
  $\xi$-invariant as well.

  Since $V$ is irreducible, precisely one of the sets $\scS_{\cD_0}$ is
  dense. If $p\not\in\overline{\scS_\cD}$ for any $\cD\neq\cD_0$ then
  one can take
  \begin{equation}
     V'=V, \qquad U=V\setminus\bigcup_{\cD\neq\cD_0}\overline{\scS_\cD}.
  \end{equation}
  Clearly $p\in U$ and the $\xi$-invariance of $\overline{\scS_\cD}$
  implies that $\O_p\subset U$ as well. By definition for any $q\in U$
  we have $\LT(I_{\Ob_q})=\cD_0$.

  If $p\in\overline{\scS_\cD}$ for some $\cD\neq\cD_0$ we replace $V$
  by $\overline{\scS_\cD}$, which is a $\xi$-invariant $\K$-variety of
  dimension strictly smaller than $\dim V$. The proof is concluded by
  induction on $\dim V$.

  If $\xi,V$ are defined over $\Q$ then we can essentially repeat the
  proof above, replacing the Zariski topology over $\K$ by $\Q$. We
  may assume without loss of generality that $V$ is irreducible over
  $\Q$. Then $V$ may be decomposed into finitely many $\K$-irreducible
  components $V_j$, which are transitively permuted by $\Gal(\K/\Q)$
  \cite[Theorem~III.4.10]{lang:ag}. Since according to
  Proposition~\ref{prop:diagrams} each of the strata $\scS_\cD$ are
  invariant under $\Gal(\K/\Q)$, it follows that precisely one of the
  sets $\scS_{\cD_0}$ is Zariski dense in $V$ over $\Q$. The rest of
  the proof can be concluded verbatim (using the Zariski topology over
  $\Q$).
\end{proof}

\section{Diophantine geometry of $\xi$}
\label{sec:diophatine}

In this section we recall some preliminary results on diophantine
geometry and apply them to study the geometry of $\xi$. We fix $\xi,V$
as in~\secref{sec:intro}, and in this section we also assume that they
are defined over $\Z$. When we use the asymptotic notation $O(\cdot)$
the constants may depend on $\xi,V$.

\subsection{Heights}

Let $P\in\Z[x_1,\ldots,x_n]$ be non-zero. We define the \emph{height}
$H(P)$ to be the maximum of the absolute values of the coefficients of
$P$, and the logarithmic height $h(P):=\log H(P)$. We denote by
$\norm{\cdot}$ the standard $L_2$-norm on $\Q[x_1,\ldots,x_n]$ with
respect to the standard monomial basis, and the standard Hermitian
norm on $\C^n$. We require some basic inequalities concerning heights
as follows.

\begin{Lem}[\protect{cf.~\cite[Lemma~1.2]{nesterenko:ind-measure}}] \label{lem:height-prod}
  Let $P_1,\ldots,P_s\in\Z[x_1,\ldots,x_n]$ and $P=P_1\cdots P_s$.
  Then
  \begin{equation}
    h(P) \le h(P_1)+\cdots+h(P_s)+O(\deg P)
  \end{equation}
\end{Lem}
\begin{proof}
  We remark that in~\cite{nesterenko:ind-measure} a slightly different
  notion of height it used: if $c(Q)$ denotes the greatest common
  divisor of the coefficients of $Q$ then Nesterenko uses the height
  function $\hat h(Q):=h(Q/c(Q))$. The conclusion follows from
  \cite[Lemma~1.2]{nesterenko:ind-measure} since
  \begin{multline}
    h(P) = \log c(P)+\hat h(P) \le \sum_j \log c(P_j) + \sum_j \hat h(P_j) +O(\deg P)\\
    \le h(P_1)+\cdots+h(P_s)+O(\deg P).
  \end{multline}
\end{proof}

\begin{Lem}\label{lem:xi-height}
  Let $P\in\Z[x_1,\ldots,x_n]$. Then
  \begin{equation}
    h(\xi^k P) \le h(P)+O(k\log(\deg P+k)).
  \end{equation}
\end{Lem}
\begin{proof}
  We first note that for every $Q\in\Z[x_1,\ldots,x_n]$ and for
  $i=1,\ldots,n$ we have
  \begin{align}
    H(\pd{}{x_i}Q)&\le(\deg Q)H(Q) \\
    H(\xi_i Q) &= O(H(Q)) \\
    h(\xi Q) &\le h(Q)+O(1+\log\deg Q). \label{eq:h-xi-Q}
  \end{align}
  where for the second estimate we note that each monomial in
  $\xi_i Q$ is given by a sum of $O(1)$ terms, each bounded by
  $O(H(Q))$, and the third estimate follows from the previous two.
  Finally, since $\deg\xi^kP=\deg P+O(k)$ we have by repeated
  application of~\eqref{eq:h-xi-Q} that
  \begin{equation}
    \begin{aligned}
      h(\xi^k P)\le &h(P)+kO(1+\log(\deg P+O(k)))\\
      =&h(P)+O(k\log(\deg P+k)).
    \end{aligned}
  \end{equation}
\end{proof}

\begin{Lem}\label{lem:height-det}
  Let $A$ be a $\rho\times \rho$ matrix with entries
  $A_{ij}\in\Z[x_1,\ldots,x_n]$ satisfying $\deg(A_{ij})\le d$ and
  $h(A_{ij})\le h$. Write $M=\det A$. Then
  \begin{equation}
    \deg M\le \rho d, \qquad h(M)\le \rho h+O(\rho d)+\rho\log\rho.
  \end{equation}
\end{Lem}
\begin{proof}
  We expand $M=\sum_\alpha S_\alpha$ as a sum of $\rho!$ summands
  $S_\alpha$, each a product of $\rho$ entries of $A$. Then
  $\deg M\le\rho d$ is clear, and
  $h(S_\alpha)\le\rho h+O(\rho d)$ according to
  Lemma~\ref{lem:height-prod}. Finally, by the subadditivity of
  $H(\cdot)$ we have
  \begin{equation}
    h(M)=\log H(\sum_\alpha S_\alpha)\le\log(\rho!e^{\rho h+O(\rho d)})\le\rho h+O(\rho d)+\rho\log\rho.
  \end{equation}
\end{proof}

\begin{Lem}\label{lem:poly-eval-bound}
  Let $P\in\Z[x_1,\ldots,x_n]$ with $\deg P=d$ and $h(P)=h$. Let
  $p\in\C^n$. Then $\abs{P(p)}\le (d+1)^n e^{h(P)} \max(\norm{p}^d,1)$.
\end{Lem}
\begin{proof}
  If $P(x)=\sum_{|\alpha|\le d}c_\alpha x^\alpha$ then $c_\alpha\le e^{h(P)}$ and
  \begin{equation}
    |P(p)| \le \sum_{|\alpha|\le d} e^{h(P)} |p^\alpha|\le (d+1)^n e^{h(P)}\max(\norm{p}^d,1).
  \end{equation}
\end{proof}

\subsection{A diophantine \Lojas. inequality}
\label{sec:lojas}

In this section we recall a result of Brownawell
\cite{brownawell:null2} showing that if a collection of polynomials of
bounded height and degree are small at a point $p\in\C^n$ then $p$ is
close to a common zero of these polynomials. This is the only
diophantine-type estimate that will be used in the sequel. We remark
that an earlier slightly weaker estimate of \cite{brownawell:null1}
would also suffice for our purposes and is presented with fully
explicit constants (which may be of interest in deriving an effective
version of the results of this paper). We present a formulation which
is somewhat weaker than the precise one given in
\cite[Theorem~2.1]{brownawell:null2} but sufficient for our purposes.
Since this result plays a key role in our considerations, we present a
proof for the convenience of the reader in~\secref{sec:nest-proof}.

Let $\w^1,\w^2\in\C P^n$ and $\bar\w^1,\bar\w^2\in\C^{n+1}$ respective
representatives chosen such that their maximal coordinate has modulus
$1$. We define the projective distance between $\w^1,\w^2$ to be
\begin{equation}
  \dist(\w^1,\w^2):= \max_{0\le i<j\le n}|\bar\w^1_i\bar\w^2_j-\bar\w^1_j\bar\w^2_i|.
\end{equation}
If $W\subset\C P^n$ is a projective variety we define $\dist(\w,W)$ to
be the minimal distance between $\w$ and a point of $W$.

We denote by $\psi:\C^n\to\C P^n$ the standard embedding defined by
$(x_1,\ldots,x_n)\to(1:x_1:\cdots:x_n)$ and by $H_\infty:=\{x_0=0\}$
the hyperplane at infinity.
  
\begin{Thm}[\protect{\cite[Theorem~2.1]{brownawell:null2}}]\label{thm:lojas}
  Let $\scP:=\{P_\alpha\}\subset\Z[x_1,\ldots,x_n]$ be a collection of polynomials satisfying
  $\deg P_\alpha\le d$ and $h(P_\alpha)\le h$ for every $\alpha$.
  Denote by $W$ their common zero locus. Let $p\in\C^n$ and suppose
  $\abs{P_\alpha(p)}\le\e$ for every $\alpha$. Then
  \begin{equation}
    \log\e \ge d^n [n\log\dist\(\psi(p),\psi(W)\cup H_\infty\)-O(d+h)].
  \end{equation}
  If $V\subset\C^n$ is a fixed variety of dimension $m$ defined over
  $\Q$ and $W$ denotes the common zero locus of $\scP$ \emph{in} $V$
  then one may replace $d^n$ above by $d^m$ (where the asymptotic
  constants can depend on $V$).
\end{Thm}

\subsection{Proof of Theorem~\ref{thm:lojas}}
\label{sec:nest-proof}

We recall some results on diophantine geometry with projective ideals
following \cite{nesterenko:ind-measure}. We present all of the results
in the context of the field $\Q$, although the material can be
developed for a finite extension with minor differences. Unlike the
rest of this paper, polynomials and ideals considered in this section
are projective. Let $\Q[X]:=\Q[x_0,\ldots,x_n]$ and similarly for
$\Z[X]$. Following \cite{nesterenko:ind-measure}, for any unmixed
homogeneous ideal $I\subset\Q[X]$ one can associate the following
quantities:
\begin{enumerate}
\item The (projective) dimension of $I$ denoted $\dim I$.
\item The \emph{degree} of $I$ denoted $\deg I\in\N$.
\item The \emph{(logarithmic) height} of $I$ denoted $h(I)\in\R_{\ge0}$.
\item For each $\w\in\C P^n$ the \emph{absolute value} of $I$ at
  $\omega$, denoted $\abs{I(\w)}\in\R_{\ge0}$.
\end{enumerate}
We recall also that the exponent of a primary ideal $I$ with
associated prime $\fp=\sqrt I$ is the minimal $k\in\N$ such that
$\fp^k\subset I$.

We note that \cite{nesterenko:ind-measure} normalizes the height
$h(P)$ of a polynomial by first normalizing the coefficients to have
no common factor. We denote this normalized height by $\hat h(P)$ as
in the proof of Lemma~\ref{lem:height-prod}, and cite the results from
\cite{nesterenko:ind-measure} using this notation to avoid confusion.
Note that we always have $h(P)\le \hat h(P)$.

The following proposition shows that the latter three quantities (the
third taken under logarithm) are essentially linear under primary
decomposition.
\begin{Prop}[\protect{\cite[Proposition~1.2]{nesterenko:ind-measure}}]\label{prop:nest-linear}
  Let $I\subset\Q[X]$ be a homogeneous unmixed ideal with
  $\dim I\ge0$. Suppose $I=I_1\cap\cdots\cap I_s$ is the reduced
  primary decomposition of $I$ with $\fp_j=\sqrt{I_j}$ and $k_j$ the
  exponent of $I_j$. Let $\w\in\C P^n$. Then
  \begin{enumerate}
  \item $\sum_{j=1}^s k_j \deg \fp_j = \deg I$.
  \item $\sum_{j=1}^s k_j h(\fp_j) \le h(I)+O(\deg I)$.
  \item $\sum_{j=1}^s k_j \log\abs{\fp_j(\w)} \le \log\abs{I(\w)} +O(\deg I)$.
  \end{enumerate}
\end{Prop}

Let $P\in\Z[X]$ be homogeneous and $\w\in\C P^n$. We define the
\emph{normalized value} of $P$ at $\w$ to be
$\norm{P}_\w:=\abs{P(\bar\w)}\cdot H(P)^{-1}$ where
$\bar\w\in\C^{n+1}$ is a representative of $\w$ chosen such that its
maximal coordinate has modulus $1$. The following proposition provides
estimates on the degree, height and absolute value at a point $\w$ of
a projective hypersurface in terms of its equation.
\begin{Prop}[\protect{\cite[Proposition~1.3]{nesterenko:ind-measure}}]\label{prop:nest-principal}
  Let $P\in\Z[X]$ be homogeneous and non-zero and set $J:=(P)$. Let
  $\w\in\C P^n$. Then
  \begin{enumerate}
  \item $\deg J=\deg P$.
  \item $h(J) \le \hat h(P)+O(\deg P)$.
  \item $\log \abs{J(\w)} \le \log\norm{P}_\w+O(\deg P)$.
  \end{enumerate}
\end{Prop}

The following proposition allows us to intersect an irreducible
projective variety with a divisor with explicit estimates on the
degree, height and absolute value at a point $\w$.

\begin{Prop}[\protect{\cite[Proposition~1.4]{nesterenko:ind-measure}}]\label{prop:nest-intersect}
  Let $\fp\subset\Q[X]$ be a homogeneous prime ideal with
  $\dim\fp\ge1$. Suppose $Q\in\Z[X]$ is homogeneous and $Q\not\in\fp$.
  Then there exists a homogeneous unmixed ideal $J\subset\Q[X]$ with
  $\dim J=\dim\fp-1$ and $Z(J)=Z((\fp,Q))$ such that
  \begin{enumerate}
  \item $\deg J\le \deg\fp \cdot \deg Q$.
  \item $h(J)\le \deg\fp\cdot \hat h(Q)+h(\fp)\cdot\deg Q+O(\deg\fp\cdot\deg Q)$.
  \item For any $\w\in\C P^n$ we have
    \begin{equation} \label{eq:J-abs}
      \log \abs{J(\w)} \le \log \delta+\deg\fp\cdot \hat h(Q)+h(\fp)\cdot\deg Q+O(\deg\fp\cdot\deg Q)
    \end{equation}
    where
    \begin{equation}\label{eq:J-options}
      \delta=
      \begin{cases}
        \norm{Q}_\w & \text{if }\dist(\w,Z(\fp))<\norm{Q}_\w \\
        \abs{\fp(\w)} & \text{otherwise}
      \end{cases}
    \end{equation}
  \end{enumerate}
  The inequality~\eqref{eq:J-abs} remains true also in the case
  $\dim\fp=0$ if we formally set $J$ to be the irrelevant ideal $(X)$
  and formally set $\abs{J(\w)}=1$.
\end{Prop}

The following proposition establishes a relation between the absolute
value of an ideal $I$ at $\w$ and $\dist(\w,Z(I))$.

\begin{Prop}[\protect{\cite[Proposition~1.5]{nesterenko:ind-measure}}]\label{prop:nest-dist}
  Let $I\subset\Q[X]$ be a homogeneous unmixed ideal and
  $r:=\dim I+1$. Let $\w\in\C P^n$. Then
  \begin{equation}
    \deg I\cdot \log\dist(\w,Z(I)) \le \frac1r \log \abs{I(\w)}+ \frac1r h(I)+O(\deg I).
  \end{equation}
\end{Prop}

We are now ready to give the proof of Theorem~\ref{thm:lojas}.
\begin{proof}[Proof of Theorem~\ref{thm:lojas}.]
  For simplicity of the presentation we will prove the first statement
  of the theorem. For the second statement one can reduce to the case
  that $V$ is irreducible over $\Q$, and the proof is essentially the
  same as presented below except that we begin with the ideal
  $\fq_0=I(V)$. We leave the details to the reader.
  
  Denote by $\~P_\alpha\in\Z[x_0,\ldots,x_n]$ the homogenization of
  $P_\alpha$ and let $\tilde\scP:=\{\tilde P_\alpha\}$. Note that
  \begin{equation}
    \hat h(\tilde P_\alpha) = \hat h(P_\alpha) \le h.
  \end{equation}
  Let $\w=\psi(p)$ and denote by $\tilde W$ the Zariski closure of
  $\psi(W)$. We remark that the common zeros of $\tilde\scP$ lie in
  $\tilde W\cup H_\infty$.

  Let $\fq_0:=(0)$. If $W=Z(\fq_0)$ there is nothing to prove.
  Otherwise we choose a polynomial $P_1\in\tilde\scP\setminus\fq_0$.
  Applying Proposition~\ref{prop:nest-principal} we obtain an ideal
  $J_1$ satisfying
  \begin{equation}
    \begin{aligned}
      \deg J_1 &\le O(d), &
      h(J_1) &\le O(d+h), &
      \log \abs{J_1(\w)} &\le \log\norm{P_1}_\w+O(d).
    \end{aligned}
  \end{equation}
  Suppose there exists an associated prime $\fq_1$ of $J_1$ and a
  polynomial $P_2\in\tilde\scP\setminus\fq_1$ such that
  \begin{equation}
    \dist(\w,Z(\fq_1)) < \norm{P_2}_\w.
  \end{equation}
  Then applying Proposition~\ref{prop:nest-intersect} to $\fq_1,P_2$
  and using Proposition~\ref{prop:nest-linear} we obtain an ideal
  $J_2$ satisfying
  \begin{equation}
    \begin{aligned}
      \deg J_2 &\le O(d^2), &
      h(J_2) &\le O(d^2+hd), &
      \log \abs{J_2(\w)} &\le \log\norm{P_2}_\w+O(d^2+hd).
    \end{aligned}
  \end{equation}
  Continuing in this manner we finally obtain a polynomial
  $P_s\in\tilde\scP\setminus\fq_{s-1}$ and $J_s$ with
  $\dim J_s=n-s$ such that
  \begin{equation}\label{eq:js-estimates}
    \begin{aligned}
      \deg J_s &\le d^s, &
      h(J_s) &\le O(d^s+hd^{s-1}), &
      \log \abs{J_s(\w)} &\le \log\norm{P_s}_\w+O(d^s+hd^{s-1})
    \end{aligned}
  \end{equation}
  and such that for any of the associated primes $\fq$ of $J_s$, and
  any $P\in\tilde\scP\setminus\fq$ we have
  $\dist(\w,Z(\fq)) \ge \norm{P}_\w$.

  We construct a rooted tree $T$ whose vertices $v$ are homogeneous
  prime ideals $\fp_v\subset\Q[x_0,\ldots,x_n]$ with associated
  multiplicities $k_v\in\N$. The root is $\fp_{s-1}$, and its children
  are the primary components of $J_s$ with their associated
  multiplicities (or just $J_s$ with multiplicity $1$ if $J_s$ is the
  irrelevant ideal). The tree is constructed recursively as follows.
  \begin{itemize}
  \item If $v$ is a vertex and $Z(\fp_v)\subset W\cup H_\infty$ or
    $\fp_v$ is the irrelevant ideal we declare it a leaf.
  \item Otherwise we choose an arbitrary polynomial $P_v\in\tilde\scP$
    satisfying $P_v\not\in\fp_v$. We remark that by construction we
    automatically have
    \begin{equation}\label{eq:good-pv-dist}
      \dist(\w,Z(\fp_v)) \ge \norm{P_v}_\w.
    \end{equation}
    We let $J_v$ denote the ideal constructed in
    Proposition~\ref{prop:nest-intersect} and $\fp_j,k_j$ the
    components of its primary decomposition. We define the
    children of $v$ to be the $\fp_j$ with associated multiplicities
    $k_v k_j$. If $J_v$ is the irrelevant ideal we define $J_v$ itself
    to be the single child of $v$, with multiplicity $1$.
  \end{itemize}

  Denote by $T_r$ the set of vertices whose ideals have dimension
  $n-r$ in $T$ and by $\cL$ the set of leafs, excluding irrelevant
  ideals. A simple induction starting with~\eqref{eq:js-estimates} and
  using Propositions~\ref{prop:nest-linear}
  and~\ref{prop:nest-intersect} gives for $r=s,\ldots,n$
  \begin{align}
    \sum_{v\in T_r} k_v \deg \fp_v &\le d^r \\
    \sum_{v\in T_r} k_v h(\fp_v) &\le O(h d^{r-1} +d^r).
  \end{align}
  Moreover, recursive expansion of $\abs{J_s(\w)}$
  in~\eqref{eq:js-estimates} using Propositions~\ref{prop:nest-linear}
  and~\ref{prop:nest-intersect}, taking into
  account~\eqref{eq:good-pv-dist}, gives
  \begin{equation}
    \log\norm{P_s}_\w \ge \sum_{v\in\cL} k_v\log\abs{\fp_v(\w)}-O(d^{n+1}+hd^n).
  \end{equation}
  We note that for any $v\in\cL$ we have
  $\dist(\w,Z(\fp_v))\le\dist(\w,\tilde W\cup H_\infty)$. Thus
  applying Proposition~\ref{prop:nest-dist} we have
  \begin{equation}
    \log\norm{P_s}_\w \ge n d^n\cdot \log\dist(\w,\tilde W\cup H_\infty)-O(d^{n+1}+hd^n).
  \end{equation}
  Finally noting that $\norm{P_s}_\w \le\e$ gives the statement
  of the theorem.
\end{proof}

\subsection{The $\mu$-elimination minors}
\label{sec:mu-minors}

Let $\cD\subset\N^n$ be an ideal. We denote by
$\cR^\cD\subset\Q[x_1,\ldots,x_n]$ the $\Q$-span of all monomials
lying \emph{outside} $\cD$, and by
$\cR^\cD_{\le d}:=\cR^\cD\cap\cP_{\le d}$. We write
$\rho=\rho(d):=\dim_\Q \cR^\cD_{\le d}$. It is known that
$\rho\sim d^\kappa$ where $\kappa$ is a natural number (see
\cite[Proposition~12.2]{eisenbud:ca} and the paragraph that follows
it; in fact $\rho$ agrees with a polynomial in $d$ for sufficiently
large $d$). In particular if $\cD=\LT(I)$ for an ideal
$I\subset\Q[x_1,\ldots,x_n]$ then $\kappa=\dim Z(I)$ (see
\cite[Theorem~12.4]{eisenbud:ca}. We assume that $\kappa\ge1$.

Let $\mu:=\mu(d)\in\N$ be an arbitrary function satisfying
$\mu(d)>\rho(d)$. Below when $d$ is clear from the context we write
$\mu$ for $\mu(d)$ to simplify the notation. We define a linear map
\begin{equation}
  T^\cD_d:\cR^\cD_{\le d}\to \Q[x_1,\ldots,x_n]^{\oplus (\mu+1)}, \qquad T^\cD_d(P)=(P,\xi P,\ldots,\xi^\mu P).
\end{equation}
We represent $T^\cD_d$ with respect to the monomial basis of
$\cR^\cD_{\le d}$ as a $\rho\times(\mu+1)$ matrix with entries in
$\Z[x_1,\ldots,x_n]$. We let $\scM^\cD_d$ denote the set of all
top-dimensional minors of this matrix, called \emph{$\mu$-elimination
  minors}. The following is obvious.

\begin{Lem}\label{lem:mu-minor-vanishing}
  Let $p\in\C^n$. Then $p$ is a common zero of $\scM^\cD_d$ if and
  only if there exists a non-zero $P\in\cR^\cD_{\le d}$ such that
  $P,\ldots,\xi^\mu P$ vanish at $p$.
\end{Lem}

Next, we estimate the degree and height of the minors in $\scM^\cD_d$.

\begin{Lem}\label{lem:minor-deg-height}
  For every $M\in\scM^\cD_d$, we have
  \begin{equation}
    \begin{aligned}
      \deg M &\le O(d^\kappa \mu) \\
      h(M) &\le O(d^\kappa \mu\log\mu)
    \end{aligned}
  \end{equation}
\end{Lem}
\begin{proof}
  Recall the $M$ is the determinant of a $\rho\times\rho$ matrix $A$
  whose entries are of the form $A_{\alpha,k}:=\xi^k x^\alpha$ where
  $k\le\mu$ and $x^\alpha$ is a monomial of degree bounded by $d$.
  Using Lemma~\ref{lem:xi-height} we have
  \begin{equation} \label{eq:A-deg-height}
    \begin{aligned}
      \deg A_{\alpha,k} &\le O(\mu) \\
      h(A_{\alpha,k}) &\le O(\mu \log\mu).
    \end{aligned}
  \end{equation}
  Recalling that $\rho\sim d^\kappa$ and using Lemma~\ref{lem:height-det}
  gives the statement of the lemma.
\end{proof}

Finally we show that if a minor in $\scM^\cD_d$ is large at a point,
then every polynomial from $\cR^\cD_d$ has a large $\xi$-derivative
there. Below we use the notation $\log_0 x:=\max(\log x,0)$.

\begin{Lem}\label{lem:minor-v-abs}
  Let $p\in\C^n,M\in\scM^\cD_d$ and suppose that $\abs{M(p)}=\e>0$.
  Then for every polynomial $P\in\cR^\cD_d$ there exists an
  index $0\le k\le\mu$ such that
  \begin{equation}
    \log \abs{\xi^k P(p)} \ge \log\norm{P}+\log\e-O(d^\kappa\mu(\log\mu+\log_0\norm{p}))
  \end{equation}
\end{Lem}
\begin{proof}
  We use the notations of the proof of
  Lemma~\ref{lem:minor-deg-height}. Denote by $A_p$ the matrix
  obtained by evaluating the entries of $A$ at $p$, so that
  $\abs{\det A_p}=\e$. We may suppose that $\norm{P}=1$. Let
  $\delta=\max_{\{k_j\}}\abs{\xi^{k_j}P(p)}$ where $\{k_j\}$ is the
  set of derivatives appearing in the matrix $A$ defining $M$. Then
  $\norm{A_p P}\le \sqrt\rho \delta$ so
  $\norm{A_p^{-1}}\ge (\sqrt\rho \delta)^{-1}$.

  On the other hand using~\eqref{eq:A-deg-height} and
  Lemma~\ref{lem:poly-eval-bound} we have
  \begin{equation}
    \norm{A_p}\le e^{O(\mu\log\mu)}\max(\norm{p},1)^{O(\mu)}.
  \end{equation}
  The following inequality (see e.g. \cite[(1.1)]{schaffer:norms})
  follows easily from the singular value decomposition of $A$
  \begin{equation}
    \norm{A_p^{-1}} |\det A_p| \le \norm{A_p}^{\dim A-1}.
  \end{equation}
  Using $\dim A=\rho\sim d^\kappa$ we have
  \begin{equation}
    \begin{aligned}
      \rho^{-1/2}\delta^{-1} \e &\le  \norm{A_p^{-1}} |\det A_p| \le \norm{A_p}^{\dim A-1} \\
      &\le e^{O(d^\kappa \mu\log\mu)}\max(\norm{p},1)^{O(d^\kappa \mu)}
    \end{aligned}
  \end{equation}
  from which it follows that
  \begin{equation}
    \delta > e^{-O(d^\kappa \mu\log\mu)}\max(\norm{p},1)^{-O(d^\kappa \mu)} \e
  \end{equation}
  as claimed.
\end{proof}

\section{Proof of the main theorem}
\label{sec:main-proof}

In this section we give a proof of the main theorem. We fix $\xi,V$ as
in~\secref{sec:intro}, with $\K=\Qa$.

\subsection{Universal lower bounds for $\xi$-derivatives}

The following theorem is the main technical ingredient in the proof of
the main theorem.

\begin{Thm}\label{thm:derivative-bound}
  Assume $\xi$ has constructible orbits in $V$. Fix $p\in V$ and
  denote $m:=\dim V$ and $\kappa:=\dim\Ob_p$. There exists a constant
  $C_p$ such that for any polynomial $P\in\cP_{\le d}$ there exists a
  polynomial $R\in\cP_{\le d}$ satisfying
  $R\rest{\Ob_p}\equiv P\rest{\Ob_p}$ and an index
  $0\le k\le C_p d^\kappa$ such that
  \begin{equation}
    \log \abs{\xi^k R(p)} \ge \log\norm{R}-O(d^{2\kappa(m+1)}\log d).
  \end{equation}
  Here the asymptotic constants may depend on $p$ as well.
\end{Thm}
\begin{proof}
  Recall the notations of~\secref{sec:mu-minors}. Denote
  $\cD:=\LT(I_{\Ob_p})$. We will in fact prove a somewhat refined
  conclusion: that one can choose $R$ in the conclusion of the theorem
  such that $R\in\cR^\cD_{\le d}\subset\cP_{\le d}$.

  We first reduce to the case where $\xi$ is defined over $\Q$.
  Suppose that $\xi$ is defined over a finite extension $\L$ of $\Q$.
  By the primitive element theorem \cite[Theorem~14]{jacobson} we may
  have $\L=\Q(\alpha)$ for some $\alpha\in\L$. We introduce an extra
  variable $y$ and let $\xi_0$ denote the $\Q$-vector field obtained
  from $\xi$ by expressing each $\L$-coefficient of $\xi$ as a
  $\Q$-rational function in $\alpha$ and substituting $y$ for
  $\alpha$, and letting $\xi_0(y)=0$. By construction $\{y=\alpha\}$
  is invariant under $\xi_0$ and $\xi_0\rest{y=\alpha}\equiv\xi$, so
  that the orbit structure of $\xi_0$ on $\{y=\alpha\}$ is the same as
  that of $\xi$.
  
  We let $D\in\Z[y]$ denote the common denominator of the coefficients
  of $\xi_0$ so that $\xi_1=D\cdot\xi_0$ has polynomial coefficients
  over $\Z$. Note that $D(\alpha)\neq0$ by construction, and hence
  $\xi_0,\xi_1$ are non-zero constant multiples of each other on
  $\{y=\alpha\}$ and thus have the same orbit structure there. We set
  $p_1:=(p,\alpha)$ and $V_1:=V\times\{\alpha\}$. In particular,
  $p_1\in V_1$ and $\xi_1$ has constructible orbits in $V_1$.

  We claim that it will suffice to prove the claim for $\xi_1,V_1$. We
  first note that $m,\kappa$ are the same in this case. We choose the
  monomial ordering $\prec$ such that $y$ precedes $x_1,\ldots,x_n$.
  Then since $I_{\Ob_{p_1}}$ contains the polynomial $y-\alpha$, its
  Gr\"obner diagram $\cD_1$ then contains all monomials containing $y$.
  It also contains the Gr\"obner diagram of
  $I_{\Ob_p}\subset I_{\Ob_{p_1}}$. Let $R_1\in\cR^{\cD_1}_{\le d}$
  and $k$ be as in the conclusion of the theorem. Then in fact
  $R_1\in\cR^\cD_{\le d}\cap\C[x_1,\ldots,x_n]$, and
  \begin{multline}
    \log \abs{\xi^k R_1(p)} = \log \abs{\xi_0^k R_1(p,\alpha)} = \log \abs{D(\alpha)^{-k} \xi_1^k R_1(p,\alpha)} \\
    \ge -k\log|D(\alpha)|+\log\norm{R_1}-O(d^{2\kappa(m+1)}\log d).
  \end{multline}
  Using $k=O(d^\kappa)$ and $|D(\alpha)|=O(1)$ we get the conclusion
  of the theorem for $\xi,V$.
  
  Next, we reduce to the case that $V$ is also defined over $\Q$.
  Suppose $V$ is defined over a finite extension of $\L$ of $\Q$. The
  Galois conjugates $\sigma V$ for $\sigma\in\Gal(\L/\Q)$ are also
  $\xi$-invariant,
  \begin{equation}
    \xi (I_{\sigma V}) = \xi (\sigma I_V) = \sigma(\xi I_V)\subset\sigma I_V=I_{\sigma V}
  \end{equation}
  where we used our assumption that $\xi$ is defined over $\Q$ and
  hence commutes with $\sigma$. The union of the (finitely many)
  Galois conjugates of $V$ is defined over $\Q$ \cite[III.4, Theorem
  10]{lang:ag} and $\xi$-invariant, and it will suffice to prove the
  conclusion with $V$ replaced by this union.
     
  By Proposition~\ref{prop:V-grobner-reduct}, after replacing $V$ by
  its $\Q$-subvariety $V'$, we may assume without loss of generality
  that there exists a Zariski open $U\subset V$ containing $\O_p$ such
  that the Gr\"obner diagram $\LT(I_{\Ob_q})=\cD$ for $q\in U$. Since
  the dimension and degree of an affine variety are determined by its
  Gr\"obner diagram, we see that $\dim \Ob_q=\kappa$ and $\deg \Ob_q$ is
  constant for $q\in U$. In the notations of
  Theorem~\ref{thm:mult-refined} we have $\Ob_q=\overline{\gamma_q}$
  and we may therefore choose a constant $C$ such that for every
  $q\in U$ and every $Q\in\cP_{\le d}$, if $Q\rest{\gamma_q}\not\equiv0$
  then $\mult_{\gamma_q} Q \le C d^\kappa$. Set $\mu(d):=C d^\kappa$.

  If $Q\in\cR^\cD_{\le d}$ is non-zero and $q\in U$ then since
  $\supp Q\cap\cD=\emptyset$ we have $Q\not\in I_{\Ob_q}$. Thus by the
  above $\mult_{\gamma_q}Q\le\mu$. If we denote $W:=Z(\scM^\cD_d)$
  then it follows from Lemma~\ref{lem:mu-minor-vanishing} that
  $W\cap U=\emptyset$. Since $U$ is open, we deduce that
  $\dist\(\psi(p),\psi(W)\cup H_\infty\)$ is bounded from below by a
  constant \emph{independent of $d$}. By
  Lemma~\ref{lem:minor-deg-height} we also have for every
  $M\in\scM^\cD_d$
  \begin{equation}
    \begin{aligned}
      \deg M &\le O(d^{2\kappa}) \\
      h(M) &\le O(d^{2\kappa}\log d)
    \end{aligned}
  \end{equation}
  Applying now Corollary~\ref{thm:lojas} we see that there exists a
  minor $M\in\scM^\cD_d$ such that $\e:=\log\abs{M(p)}$ satisfies
  \begin{equation}
    \log\e\ge -O(d^{2\kappa(m+1)}\log d).
  \end{equation}
  
  Recall that by Proposition~\ref{prop:grobner-division} we may choose
  a polynomial $R\in\cP_{\le d}$ satisfying
  $\supp R\subset \LT(\cP_{\le d})\setminus\cD$, i.e. $R\in\cR^\cD_{\le d}$, and
  $R\rest{\Ob_p}\equiv P\rest{\Ob_p}$. Lemma~\ref{lem:minor-v-abs} now
  implies that there exists $0\le k\le\mu$ such that
  \begin{equation}
    \begin{aligned}
      \log \abs{\xi^k R(p)} &\ge \log\norm{R}+\log\e-O(d^{2\kappa}(\log d+\log_0 \norm{p}) \\
      &\ge \log\norm{R}-O(d^{2\kappa(m+1)}\log d)
    \end{aligned}
  \end{equation}
  as claimed.
\end{proof}

\subsection{Lower bounds and zero counting for parametrized $\xi$-trajectories}

We begin with a simple lower bound for the maximum modulus of a
polynomial on a parametrized trajectory in terms of its
$\xi$-derivatives.

\begin{Lem}\label{lem:traj-max-estimate}
  Let $\phi:\bar D\to V$ be a parametrized $\xi$-trajectory and $z_0\in D$,
  and set $p:=\phi(z_0)$. Then for any polynomial $P\in\cP$
  and any $k\in\N$ we have
  \begin{equation}\label{eq:modulus-v-xi}
    \log \max_{\bar D} \abs{P\circ\phi} \ge \log \abs{\xi^k P(p)}-O(k\log k)
  \end{equation}
  where the asymptotic constants may depend on $p$.
\end{Lem}
\begin{proof}
  If $p\in\Sing\xi$ there is nothing to prove. Otherwise one can solve the
  differential equation $\dot x=\xi(x)$ and obtain a holomorphic function
  \begin{equation}
    \hat\phi:D_r\to V, \qquad \hat\phi'(z)=\xi(\phi(z)), \quad \hat\phi(0)=p.
  \end{equation}
  Restricting $r$ further we may assume that
  $\hat\phi(D_r)\subset\phi(D)$. It will thus suffice to prove~\eqref{eq:modulus-v-xi}
  with $\hat\phi$ and $\bar D_r$ in place of $\phi$ and $\bar D$.

  Let $P\in\cP$ and denote $f=P\circ\hat\phi$ and $M_P:=\max_{\bar D_r}\abs{f}$. By definition
  we have $f^{(k)}(0)=\xi^k P(p)$. Thus by the Cauchy estimate
  \begin{equation}
    \log\abs{\xi^k P(p)} =\log\abs{f^{(k)}(0)} \le \log M+\log(k!)-k\log r
  \end{equation}
  which, since $r$ is fixed, implies
  \begin{equation}
    \log M \ge \log\abs{\xi^k P(p)}-O(k\log k).
  \end{equation}
\end{proof}

We recall the following consequence of Jensen's formula
\cite{iy:jensen}.

\begin{Prop}\label{prop:jensen}
  Let $r>1$ and let $f:\bar D_r\to\C$ be a holomorphic function.
  Denote by $M$ (resp. $m$) the maximum of $\abs{f}$ on $\bar D_r$ (resp.
  $\bar D$). There exists a constant $C_r$ such that
  \begin{equation}
    \#\{z\in \bar D:f(z)=0\} \le C_r \log\frac M m.
  \end{equation}
\end{Prop}

We are now ready to complete the proof of Theorem~\ref{thm:main}.

\begin{proof}[Proof of Theorem~\ref{thm:main}.]
  Fix a point $p\in\phi(D)\setminus\Sing\xi$. Then
  $\phi(\bar D)\subset\O_p$, and by analytic continuation along the
  trajectory of $\xi$ we see that a polynomial vanishes on
  $\phi(\bar D)$ if and only if it vanishes on $\O_p$. Thus the
  Zariski closure of $\phi(\bar D)$ is $\Ob_p$ and in particular
  $\kappa(\phi)=\dim\Ob_p$. Let $P\in\cP_d$. According to
  Theorem~\ref{thm:derivative-bound} there exists a polynomial
  $R\in\cP_{\le d}$ satisfying $R\rest{\Ob_p}\equiv P\rest{\Ob_p}$ and
  an index $0\le k\le C_p d^\kappa$ such that
  \begin{equation}
    \log \abs{\xi^k R(p)} \ge \log\norm{R}-O(d^{2\kappa(m+1)}\log d).
  \end{equation}
  In particular, we have $R\circ\phi\equiv P\circ\phi$, and it will
  thus suffice to count the zeros of $R\circ\phi$ in $\bar D$. According
  to Lemma~\ref{lem:traj-max-estimate} we have
  \begin{equation}
    \begin{aligned}
      \log \max_{\bar D} \abs{R\circ\phi} &\ge \log \abs{\xi^k R(p)}-O(k\log k) \\
      &\ge \log\norm{R} - O(d^{2\kappa(m+1)}\log d).
    \end{aligned}
  \end{equation}
  On the other hand since $\phi$ is holomorphic in a neighborhood of
  $\bar D$ it is certainly holomorphic in a slightly larger disk $\bar D_r\supset D$.
  In particular $\phi$ is bounded in $\bar D_r$ and
  \begin{equation}
    \log \max_{\bar D_r}\abs{R\circ\phi} \le \log\norm{R} + O(d\log d).
  \end{equation}
  Finally, by Proposition~\ref{prop:jensen} we have
  \begin{multline}
    \#\{z\in \bar D:R\circ\phi=0\} \le \\
      C_r \big[ \log\norm{R} + O(d\log d) - \log\norm{R} + O(d^{2\kappa(m+1)}\log d) \big] \\
      \le O(d^{2\kappa(m+1)}\log d)
  \end{multline}
  as claimed.
\end{proof}

\section{Systems with constructible orbits}
\label{sec:examples}

In this section we give several examples of systems of differential
equations having constructible orbits (and some that do not).

\subsection{Linear systems}
\label{sec:linear}

\subsubsection{Background.}

Our main motivating example for the notion of constructible orbits
comes from the work of Nesterenko \cite{nesterenko:galois} on
transcendence properties of E-functions. A key element of Nesterenko's
approach is the following result (\cite[Theorem~3]{nesterenko:galois}
and the corollary following it): if a linear differential equation
admits a solution satisfying no non-trivial algebraic relations, then
all solutions in a Zariski open neighborhood enjoy the same property.
In other words, after removing a proper closed set, the orbit closure
relation becomes trivial (every two points are related). The general
notion of constructible orbits introduced in this paper is meant to
provide a more uniform approach for the case where solutions may
satisfy certain algebraic relations.

Nesterenko's proof of the aforementioned result is based on linear
differential Galois theory. Nesterenko
\cite[Theorem~2]{nesterenko:galois} establishes a bijection between
the algebraic $\xi$-invariant sets and varieties in a certain
projective space which are invariant under the action of the
differential Galois group. One can then reduce the study of
$\xi$-orbits to the study of orbits of this algebraic group.

Below we essentially follow Nesterenko's approach to prove that any
linear system of differential equations over $\C(t)$ admits
constructible orbits (see Theorem~\ref{thm:lin-const-orbits} for the
precise statement). However to treat this more general case we find it
technically convenient to replace Nesterenko's use of homogeneous
unmixed ideals by the related concept of Chow varieties and Chow
forms. We recall the necessary background on the Galois theory of
linear differential equations and on Chow varieties
in~\secref{sec:lin-galois} and~\secref{sec:chow} respectively.

\subsubsection{Setup.}

Consider a system of linear differential equations over $\K(t)$,
\begin{equation}\label{eq:lin-sys}
  \dot y = A(t) y, \qquad A\in\Mat_{n\times n} (\K(t)).
\end{equation}
To fit this into our general framework we consider the ambient space
$\C^{n+1}\simeq\C_t\times\C_y^n$ where $\C_t$ denotes a copy of $\C$
with coordinate $t$ and $\C^n_y$ a copy of $\C^n$ with coordinates
$y=(y_1,\ldots,y_n)$. We sometimes embed $\C^n_y$ into a projective
space $\C P^n_y$ with homogeneous coordinates $(y_0:\cdots:y_n)$.

Let $q(t)$ denote a polynomial of minimal degree such that
$\tilde A(t):=q(t)A(t)$ has polynomial entries \emph{and}
$\tilde A(t)=0$ whenever $t$ is a pole of $A(t)$. Let
\begin{equation}\label{eq:lin-sys-xi}
  \xi = q(t) [ \pd{}t+A(t)\pd{}y ] = q(t)\pd{}t+\tilde A(t)\pd{}y.
\end{equation}
Then $\xi$ corresponds to~\eqref{eq:lin-sys} in the sense that the
graph of any (local) solution of~\eqref{eq:lin-sys} is a trajectory of
$\xi$. The requirement that $\tilde A(t)$ vanishes at the poles of
$A(t)$ implies that $\Sing\xi$ agrees with the polar locus of $A(t)$.
Without imposing this condition, $\xi$ could admit non-constant
trajectories with $t\equiv\const$. While the following
Theorem~\ref{thm:lin-const-orbits} remains true in this case as well,
such trajectories do not correspond to solutions of~\eqref{eq:lin-sys}
and we therefore prefer to rule them out as a matter of convenience.

Our main goal in this section is the following theorem.
\begin{Thm}\label{thm:lin-const-orbits}
  The vector field~\eqref{eq:lin-sys-xi} has constructible orbits
  in $\C_t\times\C_y^n$.
\end{Thm}

We first give an example illustrating that the conclusion of
Theorem~\ref{thm:lin-const-orbits} fails if we allow the linear
system~\eqref{eq:lin-sys} to depend on additional parameters.

\begin{Ex}
  Consider the differential equation and corresponding vector field in
  the ambient space $\C^3\simeq\C_t\times\C_y\times\C_a$,
  \begin{equation}
    \dot y = \frac a t y, \qquad \xi=t\pd{}t+a\pd{}y.
  \end{equation}
  For every point $p=(1,1,m/n)$ with $m/n$ a reduced fraction we have
  $\Ob_p=\{a=m/n, y^n=t^m\}$. The degrees of these varieties are not
  uniformly bounded over $p\in\C^3$, and it follows from
  Proposition~\ref{prop:orbit-ideal-deg} that $\xi$ does not have
  constructible orbits in $\C^3$.
\end{Ex}

\subsubsection{Galois theory of linear differential equations}
\label{sec:lin-galois}

We recall a few basic facts concerning linear differential equations
and their Galois groups. We fix a small disk $U\subset\C_t$ not
containing any singular points of~\eqref{eq:lin-sys}. The matrix
equation
\begin{equation}\label{eq:lin-sys-matrix}
  \dot F=A(t) F, \qquad F(t)\in\GL_n(\C)
\end{equation}
admits a solution $\Phi(t)$ whose entries are holomorphic functions in
$U$ known as a \emph{fundamental solution}. The extension
$L\supset \C(t)$ generated over $\C(t)$ by the entries of $\Phi$ is a
differential field called a \emph{Picard-Vessiot} extension. The
\emph{differential Galois group} $G:=\Gal(L/\C(t))$ is defined to be the
set of field automorphisms of $L$ over $\C(t)$ which commute with
derivation.

If $\sigma\in G$ then $\sigma(\Phi)$ is another fundamental solution
of~\eqref{eq:lin-sys-matrix} and this readily implies that
$\sigma(\Phi)=\Phi\cdot A_\sigma$ where $A_\sigma\in\GL_n(\C)$. We
will require the following fundamental theorem from the Galois theory
of linear differential equations.

\begin{Thm}[\protect{\cite[Theorem~1.27]{vdp-singer}}]\label{thm:gal-basic}
  In the notations above,
  \begin{enumerate}
  \item The map $\sigma\to A_\sigma$ is a representation which embeds
    $\Gal(L/\C(t))$ as a linear algebraic subgroup of $\GL_n(\C)$.
  \item The field of elements in $L$ invariant under $\Gal(L/\C(t))$ is
    $\C(t)$.
  \end{enumerate}
\end{Thm}

\subsubsection{Chow varieties}
\label{sec:chow}

Recall that there exists a quasi-projective variety
$\tilde\scC_{k,d}(\P^n)$ called the open Chow variety, parameterizing
subvarieties (reduced, possibly reducible) of dimension $k$ and degree
$d$ in $\P^n$ \cite[Theorem 21.2]{harris:ag}. That is, there is a
bijective correspondence between the varieties $X\subset\P^n$ (of
dimension $k$ and degree $d$) and points of
$\tilde\scC^n_{k,d}(\P^n)$, which we denote $X\to\cR_X$ (the point
$\cR_X$ is called the \emph{Chow form} of $X$). If $X$ is defined over
a field $K$ then the point $\cR_X$ is $K$-rational
\cite[p.40]{ds:chow}. There is a natural action of $\GL_{n+1}$ on
$\tilde\scC_{k,d}(\P^n)$, which we denote by $\circ$, satisfying
$g\circ\cR_X=\cR_{gX}$ for $g\in\GL_{n+1}$
\cite[Proposition~2.3]{ds:chow}.

The projective closure $\bar\scC_{k,d}(\P^n)$ of $\tilde\scC_{k,d}(\P^n)$ parameterizes all
\emph{effective cycles} of dimension $k$ and degree $d$ on $\P^n$
\cite[p.272]{harris:ag}, i.e. formal linear combinations
$\sum a_i X_i$ where $a_i$ are positive integers, $X_i\subset \P^n$
are varieties of dimension $k$ and $\sum a_i\deg(X_i)=d$. There is a
projective map
\begin{equation}
  \operatorname{prod}:\bar\scC_{k,d_1}(\P^n)\times\bar\scC_{k,d_2}(\P^n)\to\bar\scC_{k,d_1+d_2}(\P^n)
\end{equation}
mapping two effective cycles to their formal sum (in the Chow form
coordinates, this corresponds to the product of the corresponding
forms \cite[p.272]{harris:ag}).

We fix projective coordinates $y_0:\cdots:y_n$ on $\P^n$, and consider
the affine space $\A^n$ defined by $y_0\neq0$ and the corresponding
hyperplane at infinity $H_\infty=\{y_0=0\}$. For us it will be
convenient to consider a quasi-projective subvariety
$\scC_{k,d}^n\subset\tilde\scC_{k,d}(\P^n)$ parameterizing
subvarieties (reduced, possibly reducible) of $\A^n$. To show that
this is indeed a quasi-projective variety, let
$B(k,d)\subset\bar\scC_{k,d}(\P^n)$ denote the set of cycles having a
non-trivial component on $H_\infty$. Then
\begin{equation}
  B(k,d) = \cup_{j=1}^d \operatorname{prod}(\bar\scC_{k,d-j}(\P^n)\times\bar\scC_{k,j}(H_\infty))
\end{equation}
where $\bar\scC_{k,j}(H_\infty)\subset\bar\scC_{k,j}(\P^n)$ denotes
the set of cycles supported at $H_\infty$, which is closed by
\cite[Exercise~21.4]{harris:ag}. Each set in the union above is
closed, being the image of a projective variety
\cite[Theorem~4.1.7]{gkz}, so $B(k,d)$ is closed, and we finally
deduce that
\begin{equation}
  \scC_{k,d}^n = \tilde\scC_{k,d}(\P^n)\setminus B(k,d)
\end{equation}
is quasi-projective as claimed. If $X\subset\A^n$ then we
will write $\cR_X$ as a shorthand for $\cR_{\bar X}$ where $\bar X$
is the projective closure of $X$.

The embedding $\GL_n\to\GL_{n+1}:g\to\diag(1,g)$ defines an action of
$\GL_n$ on $\P^n$ which preserves $\A^n$ and restricts to the usual
action of $\GL_n$ in the standard chart $y_0=1$. By definition
$\scC_{k,d}^n$ is invariant under the $\circ$-action of $\GL_n$, and
again we have $g\circ\cR_X=\cR_{gX}$ where $g\in\GL_n$ and
$X\subset\A^n$ is an affine variety of degree $d$.

\subsubsection{The Galois group action and $\xi$-invariant varieties.}

We define an injection
\begin{equation}
  \iota:\scC^n_{k,d}(\C) \to \scC^n_{k,d}(L), \qquad \iota(\cR_X):=\Phi\circ\cR_X.
\end{equation}
If $X\subset\C^n$ is a variety of degree $d$ then $\iota(\cR_X)$
represents the variety whose underlying set consists of the graphs of
solutions $\Phi\cdot v$ of (\ref{eq:lin-sys}) for $v\in X$. In
particular we have the following equivalence.

\begin{Lem}\label{lem:iota-image}
  Let $X\subset\A^n(\C(t))$ be a subvariety of dimension $k$ and
  degree $d$. Then $X$ is $\xi$-invariant if and only if $\cR_X$ is in
  $\Im \iota$.
\end{Lem}
\begin{proof}
  Let $t_0\in U$ be generic and denote by $X_{t_0}\subset\A^n(\C)$ the
  fiber of $X$ at $t_0$. Then $X$ is invariant (in a neighborhood of
  $X_{t_0}$ and then by analytic continuation everywhere) if and only
  if for every $t\in U$ the fiber $X_t$ is obtained by parallel
  transport from $X_{t_0}$,
  \begin{equation}
    X_t = \Phi(t)\cdot\Phi(t_0)^{-1}\cdot X_{t_0},
  \end{equation}
  or in other words, if and only if
  \begin{equation}
    \cR_X = \cR_{\Phi(t)\cdot\Phi(t_0)^{-1}\cdot X_{t_0}} = \Phi(t)\circ\cR_{\Phi(t_0)^{-1}\cdot X_{t_0}} = \iota(\cR_{\Phi(t_0)^{-1}X_{t_0}}).
  \end{equation}
  This occurs if and only if $\cR_X\in\Im\iota$.
\end{proof}

Recall that by Theorem~\ref{thm:gal-basic} we have an action of $G$ on
$\A^n$ defined by the embedding $G\ni\sigma\to A_\sigma\in\GL_n(\C)$.

\begin{Prop}\label{prop:invariant-bijection}
  The map $\iota$ induces an inclusion-preserving bijection between
  \begin{enumerate}
  \item the set of subvarieties of $\A^n(\C)$ of dimension $k$ and
    degree $d$ that are invariant under the action of $G$
    on $\A^n(\C)$.
  \item the set of subvarieties of $\A^n(\C(t))$ of dimension $k$ and
    degree $d$ that are $\xi$-invariant.
  \end{enumerate}
\end{Prop}
\begin{proof}
  Let $V\subset\A^n(\C)$. By Lemma~\ref{lem:iota-image} we need to
  show that $\iota(\cR_V)\in\scC^n_{k,d}(\C(t))$ if and only if $V$ is
  invariant under the action of $G$. To see this we note that for
  every $\sigma\in G$,
  \begin{equation}
    \sigma(\iota(\cR_V)) = \sigma(\Phi\circ \cR_V) = (\Phi\cdot A_\sigma)\circ \cR_V = \Phi\circ\cR_{A_\sigma V}
    =\iota(\cR_{A_\sigma V}).
  \end{equation}
  Thus $V$ is invariant under $G$ if and only if $\iota(\cR_V)$ is
  invariant under $G$. But the latter is equivalent, by
  Theorem~\ref{thm:gal-basic}(b) applied to coordinates in some affine
  chart, to $\iota(\cR_V)\in\scC^n_{k,d}(\C(t))$.
\end{proof}

A variety $X\subset\C_t\times\C_y^n$ whose components project
dominantly to $\C_t$ may naturally be viewed as a subvariety of
$\A^n(\C(t))$. We make this identification below. Similarly, every
subvariety $X\subset\A^n(\C(t))$ naturally corresponds to a
subvariety $\tilde X\subset\C_t\times\C_y^n$ ($X$ is originally
defined as a subset of $U\times\C^n_y$ for some open dense subset
$U\subset\C$, and $\tilde X$ is the Zariski closure of this set).

We now turn to the description of the orbit closures $\Ob_p$ for
$p\in\C_t\times\C^n_y$.

\begin{Prop}\label{prop:lin-orbit-types}
  For every $p\in\C_t\times\C^n_y$ the orbit closure $\Ob_p$ takes
  one of the following two forms:
  \begin{enumerate}
  \item The singleton $\{p\}$ if $p\in\Sing\xi$. 
  \item The variety $\tilde X$ where $\cR_X=\iota(\cR_V)$ and
    $V\subset\A^n(\C)$ is the Zariski closure of an orbit $G\cdot y$
    for some $y\in\A^n(\C)$.
  \end{enumerate}
\end{Prop}
\begin{proof}
  Let $p=(t_0,y_0)$. If $p\in\Sing\xi$ then certainly
  $\O_p=\Ob_p=\{p\}$. Otherwise $t_0$ is a non-singular point
  of~\eqref{eq:lin-sys} and thus $\Ob_p$ is an irreducible
  $\xi$-invariant variety projecting dominantly onto $\C_t$ and may
  thus be viewed as a $\C(t)$ subvariety of $\A^n(\C(t))$. Moreover
  $\Ob_p$ is the minimal such variety containing $(t_0,y_0)$. Since
  the bijection of Proposition~\ref{prop:invariant-bijection} is
  inclusion-preserving, we see that $\cR_{\Ob_p} = \iota(\cR_V)$ where
  $V$ is the smallest $G$-invariant variety containing
  $y=\Phi(t_0)^{-1}y_0$, namely the Zariski closure of $G\cdot y$.
\end{proof}

Let the Chow coordinates of $\cR_X$ in $\scC^n_{k,d}(\C(t))$ be given
as $(C_0:\cdots:C_N)$ where $C_k\in\C[t]$ are chosen to have no common
factor. Then we define the \emph{height} of $X$ denoted $\height X$ to
be the maximum among the $t$-degrees of $C_k$.

\begin{Lem}\label{lem:ht-vs-eqs}
  Let $\cR_X\in\scC^n_{k,d}(\C(t))$. Then $\deg\tilde X$ admits an
  upper bound depending only on $d=\deg_{\C(t)} X$ and $\height X$.
\end{Lem}
\begin{proof}
  A classical construction of Chow and van der Waerden
  \cite[Corollary~3.2.6]{gkz} gives a canonical expression for the
  equations of the projective closure $\bar X\subset\P^n(\C(t))$ in
  terms of its Chow coordinates. These equations are homogeneous of
  degree $d$ in the homogeneous coordinates on $\P^n$. Each
  coefficient of these equations is given by some (fixed) polynomial
  combination of the Chow coordinates of $\cR_X$. Choosing the Chow
  coordinates to be polynomial and coprime as above, we obtain a
  system of polynomial equations in $t$ with degrees bounded in terms
  of $\height X$.

  Since the equations above define $\bar X$ over $\C(t)$, passing to
  the affine chart by setting $y_0=1$ we obtain a system of equations
  (with degrees depending only on $d,\height X$) for $X$ over $\C(t)$.
  We denote by $Z$ the zero locus of these equations in
  $\C_t\times\C^n_y$. Then $Z$ consists of $\tilde X$ and possibly
  extra components projecting non-dominantly to $\C_t$.
  The degree of $Z$, and hence of $\tilde X$, can be estimated from
  above using the Bezout theorem in terms of the degrees of the
  defining equations (which depend only on $d,\height X$).
\end{proof}

We require one more lemma concerning the orbits of $G$. For
$y\in\A^n(\C)$ we denote by $V_y$ the Zariski closure of the orbit
$G\cdot y$.

\begin{Lem}\label{lem:G-orbits-const}
  The set $\{\cR_{V_y} : y\in\A^n(\C)\}$ is a $\C$-constructible
  subset of the disjoint union of finitely many Chow varieties
  $\scC^n_{k,d}$.
\end{Lem}
\begin{proof}
  Since $G$ acts as an algebraic group on $\A^n$, the relation
  $\{(y,z):z\in G\cdot y\}$ is constructible, being expressible by the
  $\C$-formula $\exists g\in G:z=g\cdot y$. It follows by an argument
  similar to the one given in the proof of
  Proposition~\ref{prop:orbit-ideal-deg} that the degrees of $V_y$ are
  uniformly bounded by a constant $D$ independent of $y$, and by
  Lemma~\ref{lem:deg-ideal-vs-variety} their ideals are therefore
  generated in some uniformly bounded degree $N$. It then follows that
  the relation $z\in V_y$ is constructible as well: it can be
  expressed in the form ``every polynomial $P$ of degree at most $N$
  vanishing on $G\cdot y$ also vanishes on $z$'', which is readily
  translated into a $\C$-formula.

  Let $\scC$ denote the disjoint union of $\scC^n_{k,d}$ for
  $k=0,\ldots,n$ and $d=1,\ldots,D$. The relation
  \begin{equation}
    \{(z,\cR_V):z\in V\}\subset\C^n\times\scC
  \end{equation}
  is constructible (in fact Zariski closed) by
  \cite[Theorem~21.2]{harris:ag}. Then $\{\cR_{V_y} : y\in\A^n(\C)\}$ is given
  by
  \begin{equation}
    \{ \cR_V\in\scC : \exists y\in\C^n\ \forall z\in\C^n : z\in V_y \iff z\in V \}
  \end{equation}
  which is expressible by a $\C$-formula hence $\C$-constructible as
  claimed.
\end{proof}

We are now ready to complete the proof of
Theorem~\ref{thm:lin-const-orbits}.

\begin{proof}[Proof of Theorem~\ref{thm:lin-const-orbits}.]
  Let $p\in\C_y\times\C^n_y$. According to
  Proposition~\ref{prop:orbit-ideal-deg} it suffices to show that
  $\deg\Ob_p$ is uniformly bounded independent of $p$. In the case
  that $p\in\Sing\xi$ this is obvious. By
  Proposition~\ref{prop:lin-orbit-types} it remains to consider the
  case $\Ob_p=\tilde X_V$ where $\cR_{X_V}=\iota(\cR_{V_y})$ and
  $V_y\subset\A^n(\C)$ is the Zariski closure of an orbit $G\cdot y$
  for some $y\in\A^n(\C)$.

  By Lemma~\ref{lem:G-orbits-const} $\{\cR_{V_y} : y\in\A^n(\C)\}$ is
  a $\C$-constructible subset of the disjoint union of finitely many
  Chow varieties $\scC^n_{k,d}$. It will be enough to show that
  $\deg\tilde X_V$ is uniformly bounded for $\cR_V$ in this set.
  Moreover since every constructible set can be stratified into
  finitely many smooth strata \cite[Theorem~5.38]{bpr:algorithms} it
  will be enough to consider $\cR_V\in M$ for some complex connected
  manifold $M\subset\scC^n_{k,d}$ with fixed $k,d$. By
  Lemma~\ref{lem:ht-vs-eqs} it will suffice to show that $\height X_V$
  is bounded uniformly over $\cR_V\in M$.

  We may assume that $M$ is contained in some affine chart $c_k\neq0$
  where $c_k$ is one of the homogeneous Chow coordinates (otherwise
  cover $M$ by finitely many such charts). Without loss of generality
  $k=0$ and we consider the affine coordinates $b_k=c_k/c_0$ on
  $\scC^n_{k,d}$. It will suffice to show that $b_k(\cR_{X_V})$ are
  all rational functions in $t$ of degree bounded uniformly over
  $\cR_V\in M$. Recall that
  \begin{equation}
    \cR_{X_V} = \iota(\cR_V) = \Phi\circ\cR_V\in\scC^n_{k,d}(\C(t))
  \end{equation}
  Assume for simplicity of the notation that $t_0\in U$ has been chosen so
  that $\Phi(t_0)$ is the identity matrix. Then
  \begin{equation}
    f_k:M\times U\to\C, \qquad f_k(\cR_V,t):=b_k(\cR_{X_V})=b_k(\Phi(t)\circ \cR_V)
  \end{equation}
  is holomorphic in a neighborhood of $M\times\{t_0\}$. Moreover since
  $\cR_{X_V}$ is a $\C(t)$-rational point of $\scC^n_{k,d}$ for every
  $\cR_V\in M$ and $b_k$ are affine coordinates, we see that
  $f_k(\cR_V,t)=b_k(\cR_{X_V})\in\C(t)$ is rational of some degree
  $\delta(\cR_V)$ as a function of $t$ for each fixed $\cR_V\in M$. By
  Corollary~\ref{cor:uniform-rat} we deduce that the degrees
  $\delta(\cR_V)$ are uniformly bounded over $\cR_V\in M$ thus
  concluding the proof.
\end{proof}

\subsection{Planar differential equations}
\label{sec:planar}

Consider a planar differential equation
\begin{equation}\label{eq:planar}
  \xi = P\pd{}x+Q\pd{}y, \qquad P,Q\in\C[x,y]
\end{equation}
and denote $m=\max(\deg P,\deg Q)$. For $p\in\C^2$, the orbit closure
$\Ob_p$ may be
\begin{itemize}
\item the singleton $\{p\}$, if $p$ is a singular point;
\item a curve, if $p$ is non-singular but lies on an invariant
  algebraic curve of $\xi$, i.e. an algebraic curve invariant under
  the flow of $\xi$;
\item or the whole plane otherwise.  
\end{itemize}

\begin{Thm}\label{thm:planar-const-orbits}
  The vector field~\eqref{eq:planar} has constructible orbits in
  $\C^2$.
\end{Thm}
\begin{proof}
  If~\eqref{eq:planar} has finitely many invariant algebraic curves
  then we are done by Proposition~\ref{prop:orbit-ideal-deg}. On the
  other hand, according to a theorem of Jouanolou \cite{jouanolou}
  (following work of Darboux) if~\eqref{eq:planar} admits at least
  $2+m(m+1)/2$ invariant algebraic curves then it admits a rational
  first integral $R$ (i.e. a rational function invariant under the
  flow of $\xi$). Then every orbit closure $\Ob_p$ is either a
  singleton (if $p\in\Sing\xi$) or an algebraic curve contained in the
  level set $\{R(x,y)=R(p)\}$, in which case $\deg\Ob_p\le \deg R$.
  Thus we are done by Proposition~\ref{prop:orbit-ideal-deg}.
\end{proof}

\subsection{The $j$-function and related systems}
\label{sec:j-func}

\subsubsection{Preliminaries on $j$ and its differential equation}

Recall that $j:\H\to\C$ is a surjective holomorphic function invariant
under the action of $\SL(2,\Z)$ on $\H$. We note also that the set of
critical values of $j$ is $\{0,1728\}$. It is known that $j$ satisfies
a third-order differential equation. To introduce this equation, we
recall first the notion of the \emph{Schwarzian derivative} $S(f)$ of
a holomorphic function $f$, defined as
\begin{equation}
  S(f) := \(\frac{f''}{f'}\)' -\frac12 \(\frac{f''}{f'}\)^2.
\end{equation}
The Schwarzian derivative satisfies the following chain-rule type
relation
\begin{equation}
  S(f\circ g) = (g')^2[S(f)\circ(g)]+S(g).
\end{equation}
The solutions of the differential equation $S(f)=0$ are given
precisely by the fractional linear transformations $f(z)=g\cdot z$ for
$g\in\SL(2,\C)$.

The $j$ function satisfies the non-linear differential equation
$\chi(j)=0$ where (see \cite[page~20]{masser:heights})
\begin{equation}
  \chi(f) :=  S(f)+R(f)(f')^2, \qquad R(f) := \frac{f^2-1968f+2654208}{2f^2(f-1728)^2}.
\end{equation}
As observed in \cite{fs:j-func}, the general solution for the
equation $\chi(f)=0$ takes the form $f(z)=j(g\cdot z)$ for
$g\in\SL(2,\C)$. Indeed, let $f(z)$ be any solution. Pick some point
$z_0$ where $f(z_0)\not\in\{0,1728\}$, i.e. a non-critical value of
$j$. Then locally around $z_0$ one can write $f(z)=j\circ\phi(z)$ for
some (locally) holomorphic lifting $\phi(z)$. But then
\begin{equation}
  \begin{aligned}
      0&=\chi(f)=\chi(j\circ\phi) = S(j\circ\phi)+R(j\circ\phi)(j'\circ\phi\cdot\phi')^2\\
      &=(\phi')^2[\chi(j)\circ\phi]+S(\phi)=S(\phi)
  \end{aligned}
\end{equation}
which implies that $\phi(z)=g\cdot z$ for $g\in\SL(2,\C)$.

\subsubsection{Constructible orbits for $\chi=0$.}

The equation $\chi(f)=0$ can be re-written in the form
\begin{equation}\label{eq:chi-rational}
  f'''=A(f,f',f''), \qquad A(f,f',f'') := R(f)(f')^3+\frac{3(f'')^2}{2f'}.
\end{equation}
We note that no solution of the original equation $\chi(f)=0$ belongs
to the polar locus $\{f=0,1728\}\cup\{f'=0\}$ since the equation
admits no constant solutions. We let
$q(f,f',f''):=f^3(f-1728)^3(f')^2$, so that $\tilde A=qA$ is
polynomial in $f,f',f''$ and vanishes on the polar locus of $A$.

We choose the ambient space to be $\C_t\times\C^3$ with the coordinates
$t,y,\dot y,\ddot y$. Let
\begin{equation}\label{eq:j-func-xi}
  \xi := q(y,\dot y,\ddot y)[\pd{}{t}+\dot y \pd{}y+\ddot y\pd{}{\dot y}+A(y,\dot y,\ddot y)\pd{}{\ddot y}].
\end{equation}
Then $\Sing\xi=\{q=0\}$. Thus $\xi$ corresponds to the equation
$\chi(f)=0$ in the sense that for any solution $f(z)$ the map
$z\to(z,f,f',f'')$ forms a parametrized trajectory of $\xi$, and every
trajectory through a non-singular point is described in this way.

We recall the following result of Nishioka \cite{nishioka:mahler}.

\begin{Thm}[\protect{\cite[Theorem on page~1]{nishioka:mahler}}]\label{thm:nishioka}
  Let $G\subset\SL(2,\C)$ be a Zariski dense subgroup and
  $D\subset\C P^1$ a $G$-invariant domain. Let $f:D\to\C$ be a
  non-constant holomorphic functions and suppose that it is
  $G$-automorphic, i.e. $f(z)=f(g\cdot z)$ for any $g\in G$. Then $f$
  satisfies no second order algebraic differential equation over
  $\C(t)$.
\end{Thm}

The following is a direct corollary.

\begin{Cor}
  The vector field \eqref{eq:j-func-xi} is defined over $\Q$ and has
  constructible orbits in $\C_t\times\C^3$.
\end{Cor}
\begin{proof}
  Let $p\in\C_t\times\C^3$. If $p\in\Sing\xi$ then $\Ob_p=\{p\}$.
  Otherwise the trajectory through $p$ is locally given as the image
  of the map $z\to(z,f,f',f'')$ where $f$ is a solution of
  $\chi(f)=0$, i.e. $f=j(g\cdot z)$ for some $g\in\SL(2,\C)$. We claim
  that in this case $\Ob_p=\C_t\times\C^3$. Otherwise, there would
  exist a non-zero polynomial $P\in\C[t,y,\dot y,\ddot y]$ vanishing
  identically on $\O_p$, and in particular $P(t,f,f',f'')\equiv0$.
  This possibility is ruled out by Theorem~\ref{thm:nishioka}, since
  $f=j(g\cdot z)$ is automorphic with respect to $g^{-1}\SL(2,\Z)g$ on
  the domain $g^{-1}\H$.
\end{proof}

\begin{Rem}
  Nishioka \cite[Theorem on page~1]{nishioka:mahler} proves that
  Theorem~\ref{thm:nishioka} remains true also if one considers
  equations over the field $\C(t,e^{at})$ where $a$ is any fixed
  complex number. One can similarly construct a vector field with
  trajectories of the form $(z,e^{az},j(z),j'(z),j''(z))$ and show
  that it has constructible orbits in $\C^5$.
\end{Rem}

\subsubsection{Geodesically independent $j$-translates.}
\label{sec:j-func-geodesic}

Let $r_1,\ldots,r_n\in\Qa(t)$ be $n$ non-constant rational functions
with algebraic coefficients and suppose that there exists some
$t_0\in\C$ such that $r_i(t_0)\in\H$ for $r=1,\ldots,n$. Following
Pila \cite{pila:modular-ax}, we will say that $r_1,\ldots,r_n$ are
\emph{$\L$-geodesically independent\footnote{Pila uses ``geodesically
    independent'' for what we call $\Q$-geodesically independent.}}
for a field $\L\subset\C$ if there exists no relation of the form
\begin{equation}
  r_i\equiv g r_j, \qquad g\in\GL(2,\L), \quad i\neq j.
\end{equation}

\begin{Prop}\label{prop:xi-translates}
  There exists a vector field $\xi$ on $\C_t\times\C^{3n}$ defined
  over $\Qa$ whose trajectories (except the singular points) locally
  coincide with the images of the maps
  \begin{equation}\label{eq:xi-translates-general-sol}
    t\to(t,\ j(a_1),j'(a_1),j''(a_1)\ ,\cdots,
    \ j(a_n),j'(a_n),j''(a_n)).
  \end{equation}
  where $a_k(t)=g_k\cdot r_k(t)$ and $g_1,\ldots,g_n\in\GL(2,\C)$.
  Moreover none of these maps have image contained in $\Sing\xi$.
\end{Prop}

\begin{proof}
  The proof of the Proposition is similar to the construction
  of~\eqref{eq:j-func-xi}. Namely, making a change of variable
  $w=r_k(t)$ and $\pd{}w=\frac{1}{r'(t)}\pd{}t$ we transforms the
  equation $\chi(f(w))=0$ into an equation $\chi_k(f(t))=0$ such that
  $\chi_k\in\C(t,f,f',f'')$ and the solutions of $\chi_k(f(t))=0$ are
  precisely the functions of the form $j(g\cdot r_k(t))$ for any
  $g\in\GL(2,\C)$. We then construct a corresponding differential
  equation
  \begin{equation}
    f'''=A_k(t,f,f',f''), \qquad A_k\in\C(t,f,f',f'')
  \end{equation}
  as in~\eqref{eq:chi-rational} (whose singular locus may also contain
  the set $\{r_k'(t)=0\}$). Finally we construct the vector field
  $\xi$ on $\C_t\times\C^{3n}$ by taking $n$ independent copies of the
  corresponding vectors fields~\eqref{eq:j-func-xi} (but all sharing
  the same time variable $t$). We leave the detailed derivation to the
  reader.
\end{proof}

We note in particular that the map
\begin{equation}
  t \to (t,\ j(r_1),j'(r_1),j''(r_1)\ ,\cdots,
    \ j(r_n),j'(r_n),j''(r_n))
\end{equation}
defined in a neighborhood of $t_0$ forms a parametrized trajectory of
the vector field $\xi$ defined in
Proposition~\ref{prop:xi-translates}.

We recall the following theorem from \cite{pila:modular-ax}.

\begin{Thm}[\protect{\cite[Theorem~1.1]{pila:modular-ax}}]\label{thm:pila-ax-lind}
  Suppose that $r_1,\ldots,r_n$ are $\Q$-geodesically independent.
  Then the $3n$ functions
  \begin{equation}
    j(r_1),j'(r_1),j''(r_1)\ ,\cdots,\ j(r_n),j'(r_n),j''(r_n)
  \end{equation}
  (defined in a neighborhood of $t_0$) are algebraically independent
  over $\C(t)$.
\end{Thm}

The following is a direct corollary.

\begin{Cor}\label{cor:xi-translates-constructible}
  Suppose that $r_1,\ldots,r_n$ are $\C$-geodesically independent.
  Then the vector field defined in
  Proposition~\ref{prop:xi-translates} has constructible orbits in
  $\C_t\times\C^{3n}$.
\end{Cor}
\begin{proof}
  Let $p\in\C_t\times\C^{3n}$. If $p\in\Sing\xi$ then $\Ob_p=\{p\}$.
  Otherwise the trajectory through $p$ is locally given as the image
  of a map~\eqref{eq:xi-translates-general-sol}. We claim that in this
  case $\Ob_p=\C_t\times\C^{3n}$. Otherwise, there would exist a non-zero
  algebraic relation over $\C(t)$ between the functions
  \begin{equation}
    j(a_1),j'(a_1),j''(a_1)\ ,\cdots,\ j(a_n),j'(a_n),j''(a_n).
  \end{equation}
  where $a_k(t)=g_k\cdot r_k(t)$ and $g_1,\ldots,g_n\in\GL(2,\C)$. But
  these functions are $\C$-geodesically independent (because
  $r_1,\ldots,r_n$ are) so this contradicts
  Theorem~\ref{thm:pila-ax-lind}.
\end{proof}

We remark that while Theorem~\ref{thm:pila-ax-lind} requires only the
assumption of $\Q$-geodesic independence of the functions
$r_1,\ldots,r_n$, in Corollary~\ref{cor:xi-translates-constructible}
we must require $\C$-geodesic independence.

\subsection{Concluding remarks}

The examples presented in this section are not meant to give an
exhaustive list of the classical differential equations satisfying the
constructible orbits hypothesis. For instance, the Weierstrass $\wp$
and $\zeta$ functions satisfy differential equation with constructible
orbits. Pila's paper \cite{pila:modular-ax} contains a functional
independence result involving the $j,\wp$ and exponential functions,
making it possible to extend the results
of~\secref{sec:j-func-geodesic} to deal with these functions as well.

We remark also that while we focus in this paper on holomorphic
solutions of differential equations, it is often interesting to
consider the behavior of solutions in domains whose boundary contain a
singularity. In a paper to appear separately we show that if a linear
differential equation admits a regular singular point with
quasi-unipotent monodromy then Theorem~\ref{thm:main} can be extended
to domains whose boundary contains the singular point. The result
applies in particular to the Gauss-Manin connection (or Picard-Fuchs
system) of algebraic families defined over $\Qa$ and their sections
(e.g. abelian integrals), showing that such functions satisfy an
analog of Theorem~\ref{thm:main} in essentially arbitrary domains. It
is reasonable to expect that a result analogous to
Corollary~\ref{cor:density} would follow in such extended domains as
well.

\appendix
\section{Taylor coefficients of rational functions}
\label{appendix:rational}

We consider a power series $f(t)\in\C[[t]]$,
\begin{equation}
  f(t)=a_0+a_1t+\cdots \qquad a_k\in\C.
\end{equation}
If $f$ is not clear from the context we will write $a_k=a_k(f)$. We
say that $f$ is rational of degree $d$ if it is the Taylor series at
$t=0$ of a rational function $R(t)$ of degree $d$.

\begin{Thm}\label{thm:rat-conds}
  There exists a set of polynomials $\scR_d$ in the variables
  $\{a_k\}_{k\in\N}$ such that $f$ is rational of degree at most $d$
  if and only if $C(a_0(f),a_1(f),\ldots)=0$ for every $C\in\scR_d$.
\end{Thm}
\begin{proof}
  Denote
  \begin{equation}
    P(t)=p_0+\cdots+p_dt^d, \qquad Q(t)=q_0+\cdots+q_dt^d.
  \end{equation}
  For $N>d$ we consider the system of equations
  \begin{equation}\label{eq:rational-approx}
    P(t) = f(t) Q(t) + O(t^N).
  \end{equation}
  This is a system of $N$ linear homogeneous conditions on the
  variables $p_i,q_i$ with coefficients depending on $a_i$. In
  particular, one can write a set of minors
  $C_\alpha^N\in\C[a_0,\ldots,a_{N-1}]$ such
  that~\eqref{eq:rational-approx} admits a non-zero solutions $P,Q$ if
  and only if $C^N_\alpha$ vanish for every $\alpha$. We take the set
  $\scR_d$ to be the union of the sets $\{C_\alpha^N\}$ for every
  $N$. If $f$ is rational of degree at most $d$ then the
  system~\eqref{eq:rational-approx} is solvable for every $N$, so all
  the conditions $C_\alpha^N$ vanish. Conversely, we suppose that
  these conditions all vanish and prove that $f$ is rational of degree
  at most $d$.

  Let $N>d$. Since the minors $C_\alpha^N$
  vanish,~\eqref{eq:rational-approx} admits a non-zero solution $P_N,Q_N$.
  Then $Q_N\not\equiv0$, because otherwise $P_N=O(t^N)$ which is
  impossible since $P_N$ must be a non-zero polynomial of degree $d<N$.
  Set $R_N:=P_N/Q_N$. Since $Q_N$ has order at most $d$ at $t=0$ we have
  from~\eqref{eq:rational-approx}
  \begin{equation}\label{eq:R_N-vs-f}
    R_N(t)-f(t) = O(t^{N-d}).
  \end{equation}
  We claim that for $i,j\ge 3d+1$ we have $R_i\equiv R_j$. Indeed,
  $R_i-R_j$ is a rational function of order at most $2d$ and
  by~\eqref{eq:R_N-vs-f} we have
  \begin{equation}
    R_i(t)-R_j(t) = O(t^{3d+1-d}) = O(t^{2d+1})
  \end{equation}
  which is possible only if $R_i-R_j\equiv0$. To conclude, if we set
  $R:=R_{3d+1}$ then~\eqref{eq:R_N-vs-f} shows that $R$ approximates
  $f$ to every order, hence $f\equiv R$ proving that $f$ is rational
  of degree at most $d$.
\end{proof}

\begin{Cor}\label{cor:uniform-rat}
  Let $M$ be a complex connected manifold, $t_0\in\C$ and $W$ an open
  neighborhood of $M\times\{t_0\}$ in $M\times\C$. Suppose that
  $f:W\to\C$ is holomorphic, and for every $p\in M$ the function
  $f_p(\cdot):=f(p,\cdot)$ is rational. Then the degrees of $f_p$ are
  uniformly bounded over $p\in M$.
\end{Cor}
\begin{proof}
  We write a Taylor expansion $f_p(t)=\sum a_k(p) (t-t_0)^k$ where $a_k$ are
  holomorphic functions given by
  \begin{equation}
    a_k(p):M\to\C, \qquad a_k(p)=\frac1{k!}\pd{^k}{t^k} f(p,t)\rest{t=t_0}.
  \end{equation}
  By Theorem~\ref{thm:rat-conds} it is enough to prove that for some
  $d$, the conditions from $\scR_d$ vanish identically in $p$. Assume
  to the contrary that for every $d$ there exists a condition
  $C_d\in\scR_d$ such that $C_d:=C_d(f_p)$ is not identically
  vanishing as a function of $p$. Then the zero locus of $C_d(f_p)$ is
  a proper analytic subset $V_d$ of $M$, and in particular is nowhere
  dense and closed. On the other hand, by assumption $f_p$ is rational
  for every $p$, and Theorem~\ref{thm:rat-conds} implies that
  $M=\cup_d V_d$ contradicting the Baire category theorem.
\end{proof}

\bibliographystyle{plain} \bibliography{nrefs}

\end{document}